\documentclass[10pt]{amsart} %%%%%%%%%%%%%%% Revisione 12.10.19 %%%%%%%%%%%%%%%%
\usepackage{a4,amsmath,amssymb,amsthm,eucal}
\usepackage{epsfig}
\newcommand\N{\mathbb{N}}
\newcommand\R{\mathbb{R}}

\newcommand{\forget}[1]{}
\newcommand{\HH}{{\mathcal H}}
\renewcommand{\H}{\HH^1}
\def\Om{{\Omega}}  %%%{{\bar{\Omega}}}
\def\om2{{\Om\times\Om}}
\def\M{{\mathcal M}}

\def\supp{\mathrm{supp}\,}

\def\eps{\varepsilon}
\def\Lip{\mathrm{Lip}}

\newcommand{\F}{\mathcal{F}}

\newcommand{\res}{\llcorner} %   {\mathop{\hbox{
\newcommand{\MM}{\mathbb M}

\usepackage{color}
\newcommand{\weakto}{\rightharpoonup}

\newcommand{\ld}{[\![}
\newcommand{\rd}{]\!]}
\newtheorem{theorem}{Theorem}[section]
\newtheorem{definition}[theorem]{Definition}

\newtheorem{lemma}[theorem]{Lemma}
\newtheorem{proposition}[theorem]{Proposition}
\newtheorem{corollary}[theorem]{Corollary}

\theoremstyle{remark}
\newtheorem{remark}[theorem]{Remark}
\newtheorem{example}[theorem]{Example}

\numberwithin{equation}{section}

\title{Decomposition of acyclic normal currents in a metric space}

\author{Emanuele Paolini}%
\address[Emanuele Paolini]{Dipartimento di Matematica ``U.~Dini'', Universit\`{a} di
Firenze, viale Morgagni 67/A, 50134 Firenze, Italy.}
\email{paolini@unifi.it}

\author{Eugene Stepanov}%
\address[Eugene Stepanov]{
St.Petersburg Branch
of the Steklov Mathematical Institute of the Russian Academy of Sciences,
Fontanka 27,
191023 St.Petersburg,
Russia
\and
Department of Mathematical Physics, Faculty of Mathematics and Mechanics,
St. Petersburg State University, Universitetskij pr.~28, Old Peterhof,
198504 St.Petersburg, Russia%, email: stepanov.eugene@gmail.com
}
\email{stepanov.eugene@gmail.com}
\thanks{The work of the second author was financed by GNAMPA, by RFBR grant \#11-01-00825,
  by the project 2008K7Z249 ``Trasporto ottimo di massa,
disuguaglianze geometriche e funzionali e applicazioni'' of the
Italian Ministry of Research,
as well as by the project ANR-07-BLAN-0235 OTARIE}

\begin{document}

\begin{abstract}
We prove that  every acyclic normal one-dimensional real Am\-bro\-sio-Kirchheim current
in a Polish (i.e.\ complete separable metric) space
 can be decomposed in curves, thus generalizing the analogous classical result proven by S.~Smirnov in Euclidean space setting.
 The same assertion is true for every complete metric space under a suitable set-theoretic assumption.
%For illustrative purposes, an application showing the possible use of this result for optimal mass transportation problem in
%metric spaces
%is also provided.
\end{abstract}

\maketitle
%\tableofcontents

\section{Introduction}

The main result of the paper is Theorem~\ref{th_decompNorm1acycl} which says, very roughly speaking, that every acyclic normal one-dimensional real current in a complete metric space
 can be decomposed in curves. By currents here we mean
 Ambrosio-Kirchheim currents introduced in~\cite{AmbrKirch00} and generalizing the classical Whitney flat chains in Euclidean space from~\cite{Feder}. For the readers' convenience we recall some basic
 facts about Ambrosio-Kirchheim currents in Appendix~\ref{sec_curr0}.

Throughout the paper we silently assume, as it is now customary when dealing with metric currents, that
the density character (i.e.\ the minimum cardinality of a dense subset) of every metric space is an Ulam number.
This guarantees that every finite positive Borel measure is tight (even Radon when the space is complete),
is concentrated on some $\sigma$-compact subset and the support of this measure is separable (see, e.g., proposition~7.2.10
from~\cite{Bogachev06}). Besides being consistent with the Zermelo-Fraenkel set theory, this assumption is also
not restrictive because, as mentioned in~\cite{AmbrKirch00}, the whole theory of metric currents could have been
developed under the additional requirement that mass measures of the currents be tight. In fact, without this assumption
our result will be proven to hold for every complete metric space when $\mu_T$ and $\mu_{\partial T}$ are tight, and hence,
in particular, for Polish (i.e.\ complete separable metric) spaces.

 In the Euclidean space setting the analogous result on
 decomposition of acyclic normal currents in curves
 has been first proven by S.~Smirnov %in~\cite{Smirnov94}
 (see theorem~C from~\cite{Smirnov94}) and further applied in many papers, especially dealing with optimal mass transportation problem.
 Our result thus generalizes the classical one from~\cite{Smirnov94} to generic metric spaces and hence opens the way to new treatment of optimal mass transportation problems in different metric structures. The technique
 of the proof we adopt is different from the one used in~\cite{Smirnov94} and hence is new also in Euclidean setting.

As an illustration of the results of the paper, in %the last section
Appendix~\ref{sec_omt1} we study a formulation of the optimal mass transportation
problem in terms of metric currents and show that in most reasonable cases of geodesic metric spaces it is equivalent to the classical Monge-Kantorovich setting, while in general it is not, and, moreover, from an applicative point of view it is
more natural for mass transportation. Although this can be proven also by other means, we think that
the use of decomposition result for metric currents
is the most natural and easy way to get it, and, moreover, such a strategy is helpful also for
different kind of optimal transportation problems. %(e.g.\ for so-called branched transportation~\cite{BernCasMor08_book}).
%In the other two appendices we recall very briefly the basic notions of the theory of metric currents as well as prove some
%technical statements regarding metric currents.

\section{Notation and preliminaries}
For metric spaces $X$ and $Y$ we denote
by $\Lip(X,Y)$ (resp.\ $\Lip_b(X,Y)$) the set of all (resp.\ bounded) Lipschitz maps $f\colon X\to Y$ with arbitrary
  Lipschitz constant.
If $Y=\R$, we will omit the reference to $Y$ and write simply %$\Lip_K(X)$,
$\Lip(X)$ and $\Lip_b(X)$ respectively.
The Lipschitz constant of a map $f\colon X\to Y$ will be denoted by $\Lip(f)$.
The supremum norm of a map $f\colon X\to \R$ is denoted by $\|f\|_\infty$.

The metric spaces considered will further be tacitly assumed to be
complete, unless explicitly stated otherwise.
%We recall that every separable metric space is isometrically embedded
%into $\ell^\infty$ (the Banach space of bounded sequences)~\cite{AmbrKirch00}.

All the measures we will consider in the sequel are signed Borel measures
with finite total variation over some metric space $E$.
The narrow topology on measures is defined by duality with the space $C_b(E)$ of continuous bounded functions.
For a set $e\subset E$ we denote by $1_e\colon E\to \R$ its characteristic function.

We recall that a Banach space $E$ is said to have
\emph{bounded approximation property}
whenever for every compact set $K\subset E$ there is a
sequence of linear operators
$\{T_n\}$, $T_n\colon E\to E$, of finite rank (i.e.\
with finite-dimensional images), such that
$\|T_n x-x\|_E\to 0$ for all $x\in K$ as $n\to \infty$, and the operator
norms of $T_n$ are bounded by a universal constant $C>0$.
If one can choose this sequence so as to have $C=1$ then
one says that $E$ has \emph{metric approximation property}.
Clearly, the above convergence is uniform in $K$.
In fact, if $\{y_\nu\}\subset K$,
$y_\nu\to y$ as
$\nu\to \infty$, then
\begin{align*}
\|T_\nu y_\nu-y\|_E &\leq \|T_\nu y_\nu - T_\nu y\|_E + \|T_\nu y -
y\|_E
%\\
%&
\leq
C \|y_\nu-y|| +\|T_\nu y - y\|_E \to 0
\end{align*}
as $\nu\to \infty$. A typical example of Banach spaces with
bounded approximation property is given by Banach spaces possessing Schauder
(topological) basis.
%Vice-versa, it is well-known that
%for every Banach space $X$ with
%bounded approximation property there is a Banach space $Y$ such that
%$X\bigoplus Y$ possesses Schauder basis.

\subsection{Curves}

We equip the set of Lipschitz curves $\theta\colon [0,1]\to E$ with the distance
\begin{equation}\label{defdt}
d_\Theta(\theta_1,\theta_2):=\inf\left\{
\max_{t\in [0,1]}
d(\theta_1(t),\theta_2(\phi(t)))\,:\,\phi\colon [0,1]\to [0,1] \mbox{ bijective increasing}
\right\},
\end{equation}
and call
two
Lipschitz-continuous curves
 $\theta_1$, $\theta_2$:
$[0,1]\to E$
\emph{equivalent}, if
\[
d_\Theta(\theta_1,\theta_2)=0.
\]
It is not difficult to show that the equivalence of $\theta_1$ and $\theta_2$ means the existence
of continuous surjective nondecreasing functions
(called usually ``reparameterizations'') $\phi_1$, $\phi_2$: $[0,1]\to [0,1]$
such that $\theta_1(\phi_1(t))=\theta_2(\phi_2(t))$ for all $t\in [0,1]$.
 The set of equivalence classes of Lipschitz curves
equipped with the distance $d_\Theta$ will be denoted $\Theta(E)$
(we will further usually omit the reference to $E$ if it is clear from the context, and write simply $\Theta$).
In this way each $\theta\in \Theta$
can be clearly identified
with some oriented
rectifiable curve. In the sequel we will frequently slightly abuse
the language, identifying the
elements of $\Theta$ (i.e.\ oriented rectifiable
curves) with their parameterizations (i.e.\ Lipschitz-continuous paths
parameterizing such curves), when it cannot lead to a confusion.
It is easy to see that $\theta_\nu\to \theta$ in $\Theta$ implies
the Hausdorff convergence of the respective traces, though the converse
is clearly not true.

We call
$\theta\in \Theta$ an arc, if it is injective.

\subsection{Ascoli-Arzel\`{a} theorem revisited}

We will need the following version of an Ascoli-Arzel\`{a} type theorem.

\begin{proposition}\label{prop_AscoliArzela_metr1}
Let $E$ be a complete metric space,
$\theta_\nu\colon [0,1]\to E$ be Lipschitz functions with
uniformly bounded Lipschitz constants $\Lip\, \theta_\nu\leq L$ and satisfying the following uniform tightness condition:
for every $\varepsilon>0$ there is a compact set $K_\varepsilon\subset E$ such that $\mathcal{L}^1(\theta_\nu^{-1}(K_\varepsilon^c))\leq \varepsilon$ for all $\nu\in \N$.
Then there is a subsequence of $\theta_\nu$ uniformly converging to some Lipschitz function $\theta\colon [0,1]\to E$.
\end{proposition}

\begin{proof}
We prove first the statement in the case when $E$ is a Banach space
with bounded approximation property. In this case, given an $\varepsilon>0$, consider a compact set $K\subset E$ such that
 $\mathcal{L}^1(\theta_\nu^{-1}(K^c))\leq \varepsilon/8L$ for all $\nu\in \N$, and a linear finite rank operator $T$ with operator norm $C>0$ such that
 \[
 \sup_{x\in K} \|x-T x\| \leq \varepsilon/6.
 \]
Denoting $\theta_\nu'(t):=T\theta_\nu(t)$, one has that
 $\theta_\nu'\colon [0,1]\to E_n$ are $CL$-Lipschitz functions
 with values in a finite dimensional subspace $E_n\subset E$. Since
 $\theta_\nu'([0,1])\cap T K\neq \emptyset$ for all $\nu\in \N$, then
 all $\theta_\nu'$ are uniformly bounded, and hence by Ascoli-Arzel\`{a} theorem there is a subsequence of $\nu$ (which we do not relabel for brevity) such that $\theta_\nu'$ are uniformly convergent.
 Let then $N\in \N$ be such that
 \[
 \sup_{t\in [0,1]} \|\theta_k'(t)-\theta_\nu'(t)\| \leq \varepsilon/6
 \]
 for all $k\geq N$ and $\nu\geq N$.
 Thus for such $k$ and $\nu$ we get
\begin{align*}
\|\theta_k(t)-\theta_\nu(t)\| & \leq \|\theta_k(t)-\theta_k'(t)\|+ \|\theta_k'(t)-\theta_\nu'(t)\| + \|\theta_\nu'(t)-\theta_\nu(t)\|\leq 3\varepsilon/6=\varepsilon/2
\end{align*}
whenever $t\in \theta_\nu^{-1}(K)\cap \theta_k^{-1}(K)$.
Minding that
\begin{align*}
\mathcal{L}^1\left(\left(\theta_\nu^{-1}(K)\cap \theta_k^{-1}(K)\right)^c\right) &=
\mathcal{L}^1\left(\theta_\nu^{-1}(K^c)\cup \theta_k^{-1}(K^c)\right)\\
&\leq
\mathcal{L}^1\left(\theta_\nu^{-1}(K^c)\right)+ \mathcal{L}^1\left(\theta_k^{-1}(K^c)\right)\leq
\varepsilon/4L,
\end{align*}
we obtain that for every $s\in [0,1]$ there is a $t\in \theta_\nu^{-1}(K)\cap \theta_k^{-1}(K)$ such that $|t-s|\leq \varepsilon/4L$. Therefore,
\begin{align*}
\|\theta_k(s)-\theta_\nu(s)\| & \leq \|\theta_k(s)-\theta_k(t)\|+ \|\theta_k(t)-\theta_\nu(t)\| + \|\theta_\nu(t)-\theta_\nu(s)\|\\
& \leq L|t-s|+ \varepsilon/2 + L|t-s|\leq \varepsilon
\end{align*}
 for all $k\geq N$ and $\nu\geq N$. We have shown therefore that the chosen subsequence of $\theta_\nu$ is %pointwise
 uniformly Cauchy,
 hence %pointwise
 uniformly converging
 %, and applying again the Ascoli-Arzel\`{a} theorem, we get uniform %convergence of this subsequence
 to a $L$-lipschitz function
 $\theta\colon [0,1]\to E$ as claimed.

 For the case when $E$ is a %generic
 complete separable metric space, recall that
 by~\cite{Assouad78} there is a bi-Lipschitz embedding $g\colon E\to c_0$,
where $c_0\subset \ell^\infty$ stands for the Banach space of vanishing
sequences, which possesses the Schauder basis and hence satisfies the bounded approximation property. It suffices then to apply the proven result
to the sequence $g\circ \theta_\nu\colon [0,1]\to c_0$, obtaining that a subsequence (not relabeled) of
$\{g\circ \theta_\nu\}$ %converges
is uniformly Cauchy, hence so is
the sequence $\{\theta_\nu\}$ (because $g^{-1}$ is Lipschitz),
 and thus the latter converges
uniformly to some $L$-Lipschitz function.
%$h\colon [0,1]\to c_0$, hence the respective $\theta_\nu$ converge
%uniformly
%to $g^{-1}\circ h\colon [0,1]\to E$ which is still a $L$-Lipschitz %function.

Finally, if $E$ is a generic complete metric space (not necessarily separable), we just recall that $\cup_\nu \theta_\nu([0,1])$ is
$\sigma$-compact, hence separable, and we may consider $\theta_\nu$
as acting into the closure $\overline{\cup_\nu \theta_\nu([0,1])}$
of the latter, and refer to the above proven case.
\end{proof}

\section{Subcurrents}
%%%%%%%%%%%%%%%%%%%%

In the sequel we will be frequently using the notion
of a subcurrent of a given current as introduced
in the definition below.

%%%%%%%%%%%%%%%%%%%%%%%%%%%%%%%%%%%%%%%%
\begin{definition}\label{def_subcurrent}
We say that  $S$ is a \emph{subcurrent} of $T$,
and write $S\leq T$, where $T$ and $S$ are $k$-dimensional currents,
whenever
\[
  \MM(T-S) + \MM(S) \le \MM(T).
\]
\end{definition}
%%%%%%%%%%%%%%%%%%%%%%%%%%%%%%%%%%%%%%%%

We now provide a series of remarks concerning the above definition.

\begin{remark}
Since the inequality
\[
  \MM(T-S) + \MM(S) \ge \MM(T)
\]
always holds true, then $S$ is a subcurrent of $T$,
if and only if the equality
actually holds.
\end{remark}

\begin{remark}\label{rem_subcurr2}
If $R\leq S$ and $S\leq T$, then $R\leq T$.
In fact,
\begin{align*}
  \MM(T) &\ge \MM(S) + \MM(T-S)
  \ge \MM(R) + \MM(S-R) + \MM(T-S)\\
  &\ge \MM(R) + \MM(T-R),
\end{align*}
because of the triangle inequality $\MM(T-R) \le
\MM(T-S)+\MM(S-R)$.
\end{remark}

\begin{remark}\label{rem_subcurr3}
Let $T$ be a current %with finite mass $\MM(T)<+\infty$
and let $e\subset E$ be a
Borel set.
Then $T \res e \leq T$.
In fact,
\[
%\begin{align*}
  \MM(T)
  %&
=\mu_T(E) =  \mu_T(e)+\mu_T(e^c) %\\
  %&
=   \MM(T\res e) + \MM(T-T\res e).
%\end{align*}
\]
\end{remark}

\begin{remark}\label{rem_subcurr4}
If %$T$ is a current with finite mass $\MM(T)<+\infty$,
$S\leq T$, then for every Borel set
$e\subset E$ one has
$S \res e \leq T\res e$.
In fact, by the triangle inequality
\begin{align*}
\MM(T\res e)&\leq \MM((T - S)\res e)+ \MM(S\res e)\\
\MM(T\res e^c) &\leq \MM((T-
S)\res e^c) + \MM(S\res e^c),
\end{align*}
 while
%summing up
%the equality holds.
if we sum the above inequalities, then as a result we get an
equality since $S\leq T$. Hence the above inequalities are in fact
equalities for all Borel $e\subset E$. In particular, this also
implies
\begin{equation}\label{eq:musum}
\mu_T = \mu_{T-S} +\mu_S,
\end{equation}
and hence $\mu_S\leq \mu_T$. On the other hand, if~\eqref{eq:musum}
holds, then $S\le T$ since
\[
  \MM(S) + \MM(T-S) = \mu_S(E) + \mu_{T-S}(E)
  = \mu_T(E) = \MM(T).
\]
\end{remark}

%%%%%%%%%%%%%%%%%%%%%%%%%%%%%%%%%%%%%%%%
\begin{lemma}\label{lm:compconv}
Let $T_\nu$ be a sequence of currents,
$S_\nu\leq T_\nu$, and suppose that both
$S_\nu\weakto S$ and $T_\nu\weakto T$ weakly as currents as $\nu\to \infty$, while
$\MM(T_\nu)\to \MM(T)$.
Then %, $\MM(T)<+\infty$ implies that
$S\leq T$ and $\MM(S_\nu)\to \MM(S)$.
\end{lemma}
%%%%%%%%%%%%%%%%%%%%%%%%%%%%%%%%%%%%%%%%

\begin{proof}
Consider the sequence $\{T_\nu-S_\nu\}$ which converges to $T-S$ in the weak sense
of currents. By the lower
semicontinuity of $\MM$ we know that
\begin{equation}\label{diseq:comp}
\begin{aligned}
\MM(S)+\MM(T-S)
&\le\liminf_{k\to\infty} \MM(S_\nu) + \liminf_{k\to\infty}\MM(T_\nu-S_\nu)\\
&\le\liminf_{k\to\infty} [\MM(S_\nu) + \MM(T_\nu-S_\nu)] \\
& \le \liminf_{k\to\infty} \MM(T_\nu)=\MM(T),
\end{aligned}
\end{equation}
i.e.\ $S\leq T$.
Since we also have $\MM(T)\le \MM(S)+\MM(T-S)$, the inequalities
in~\eqref{diseq:comp} actually are equalities. Also, since $\MM(T-S)\le
\liminf_\nu\MM(T_\nu-S_\nu)$ we obtain $\MM(S)=\liminf_\nu \MM(S_\nu)$.
This is also true for every subsequence of $S_\nu$, hence we have full
convergence of the sequence $\MM(S_\nu)$ to $\MM(S)$ as $\nu\to \infty$.
\end{proof}
%%%%%%%%%%%%%%%%%%%%%%%%%%%%%%%%%%%%%%

We give now the definition of a cycle.

%%%%%%%%%%%%%%%%%%%%%%%%%%%%%%%%%%%%%%%%
\begin{definition}
%with $\MM(T)<+\infty$.
We say that $C\in \M_k(E)$ is a \emph{cycle} of $T\in \M_k(E)$, if $C\leq T$ and
$\partial C=0$.
%(in particular $C$ is also a
%component of $T$).
We say that $T$ is \emph{acyclic}, if $C=0$ is the only cycle of $T$.
\end{definition}
%%%%%%%%%%%%%%%%%%%%%%%%%%%%%%%%%%%%%%%%

It is easy now to prove the possibility to find such a cycle of
every current $T$, % with finite mass,
that
$T-C$ is acyclic. Of course, such a representation of a current as a sum of a cycle and an acyclic current
is not unique, as can be seen, for instance, on the example of a current defined by a curve going from the south
pole of $S^2$ to the north pole along some big semicircle, then back to the south pole along another big semicircle
and finally back again to the north pole along a third big semicircle.

%%%%%%%%%%%%%%%%%%%%%%%%%%%%%%%%%%%%%%%%
\begin{proposition}\label{prop_acycl1}%%%% versione 13.07.2012
Every current $T$
contains a
cycle $C$ such that $T-C$ is acyclic.
\end{proposition}
%%%%%%%%%%%%%%%%%%%%%%%%%%%%%%%%%%%%%%%%
\begin{proof}
Define
\[
 \xi(T) := \sup \{\MM(C) \colon \text{$C$ is a cycle of $T$}\}.
\]
%Clearly, $\xi(T)>0$ unless $T$ is already acyclic (in which case there is nothing to prove).
Let $C_0$ be a cycle of $T_0:= T$ such that $\MM(C_0)\geq \xi(T_0)/2$ and let $T_1:= T_0-C_0$.
Proceeding by induction we can define a sequence of currents $C_\nu$ such that
$C_\nu$ is a cycle of $T_\nu$ with $\MM(C_\nu) \geq \xi(T_\nu)/2$ and $T_{\nu+1} := T_\nu - C_\nu$.

Let $C$ be any cycle of $T_{\nu+1} = T_\nu - C_\nu$.
Putting together $C_\nu \le T_\nu$ and $C\le T_\nu - C_\nu$,
with the use of the triangle inequality $\MM(C+C_\nu)\leq \MM(C) + \MM(C_\nu)$ we obtain
\begin{align*}
  \MM(C+C_\nu) + \MM(T_\nu - C_\nu - C) & \leq \MM(C) + \MM(C_\nu) + \MM(T_\nu - C_\nu - C) = \MM(T_\nu),
\end{align*}
hence $\tilde C := C_\nu + C$ is a cycle of $T_\nu$ and $\MM(\tilde C)=\MM(C) + \MM(C_\nu)$ (i.e.\
$C\leq \tilde C$).
This means that $\xi(T_\nu) \ge \MM(\tilde C) = \MM(C_\nu) + \MM(C)$, hence
\[
   \MM(C) \le \xi(T_\nu) - \MM(C_\nu) \le \xi(T_\nu)/2,
\]
and in particular $\xi(T_{\nu+1}) \le \xi(T_{\nu})/2$.

One has therefore that
\[
\MM(C_\nu)\leq \xi(T_\nu) \leq \frac{\xi(T_0)}{ 2^\nu},
\]
so that
$\sum_\nu C_\nu$ is convergent in mass and hence so is the sequence
$\{T_\nu\}$, since $T_\nu = T-\sum_{k=0}^\nu C_k$. Letting $T':=\lim_\nu T_\nu$, we have $T'\leq T$
by Lemma~\ref{lm:compconv} and we
claim
that $T'$ is acyclic. In fact, let $\nu\in \N$ be arbitrary. Since
$T_{\nu+k}\leq T_\nu$ (because, in fact, $T_{\nu+k}\leq\ldots\leq T_{\nu+1}\leq T_\nu$)
for all $k\in \N$, %and $\nu\in \N$,
then passing to the limit as $k\to\infty$
we get again by Lemma~\ref{lm:compconv} that $T'\leq T_\nu$.  Thus if $C'$ is a cycle of $T'$, it is also
a cycle of $T_\nu$, so that $\MM(C')\leq \xi(T_\nu)$, and since $\xi(T_\nu)\to 0$ we obtain that $C'=0$.
\end{proof}

\section{Smirnov decomposition of currents}

To each $\theta\in \Theta$ we associate the integral one-dimensional
current $\ld\theta \rd$ defined by
\[
\ld\theta \rd(f\,d\pi):=
\int_0^1
  f(\theta(t))\,d\pi(\theta(t)) =\theta_{\#}\ld 0,1 \rd (f\,d\pi)
\]
(note that the latter integral does not depend on the parameterization of $\theta$
so it is
well defined on equivalence classes $\theta\in \Theta$).
We also define the \emph{parametric length} of $\theta$ as
\[
  \ell(\theta) := \int_0^1 |\dot{\theta}(t)| \, dt.
\]
Clearly, one has
$\MM(\ld\theta\rd)\leq \ell(\theta)$,
while when $\theta$ is an arc, then
\[
\H(\theta)=\MM(\ld\theta\rd)= \ell(\theta).
\]

The following rather simple assertion is valid.

\begin{lemma}\label{lm_Fcontinuous}
If $\theta_\nu\in \Theta$ be curves with uniformly bounded length,
$\ell(\theta_\nu)\leq C <+\infty$ for all $\nu\in \N$, and
 $\theta_\nu\to \theta\in \Theta$
as $\nu\to \infty$, then
$\ld\theta_\nu\rd(f\,d\pi)\to \ld\theta\rd(f\,d\pi)$
for every $f\,d\pi\in D^1(E)$. In other words,
the map $\theta\in \Theta\mapsto \ld\theta \rd$
is a continuous embedding of each subset of curves from $\Theta$
with uniformly bounded lengths
into the space of integral
one-dimensional currents endowed with weak topology of currents.
%(or, equivalently, with the flat norm topology).
\end{lemma}

\begin{proof}
Note that
\[
\ld\theta_\nu\rd(f\,d\pi)= \ld 0,1 \rd (f\circ\theta_\nu\,d\pi\circ\theta_\nu),
\]
so that the statement follows from the basic continuity property (theorem~3.5(ii) from~\cite{AmbrKirch00}) of currents
(alternatively, recalling $\ld\theta_\nu\rd(f\,d\pi)= \int_0^1 (f\circ\theta_\nu)(x)\,d(\pi\circ\theta_\nu) (x)$,
one could have used just elementary calculus).
%referring to convergence of
%$f\circ\theta_\nu$ and of $\pi\circ\theta_\nu$ together with boundedness of variations of the latter).
\end{proof}

%\begin{remark}\label{rem_Fcont_alt}
%Alternatively, recalling that
%\[
%\ld\theta_\nu\rd(f\,d\pi)= \int_0^1 (f\circ\theta_\nu)(x)\,d(\pi\circ\theta_\nu) (x),
%\]
%one could have given an elementary calculus proof of Lemma~\ref{lm_Fcontinuous} based on convergence of
%$f\circ\theta_\nu$ and of $\pi\circ\theta_\nu$ together with boundedness of variations of the latter.
%\end{remark}

Further on we call any finite positive Borel measure $\eta$ on $\Theta$ a \emph{transport}, because
it may be interpreted, roughly speaking, as the information on the amount of mass transported over each curve $\theta\in\Theta$.
Given a transport $\eta$ on $\Theta$ we define a
functional $T_\eta$ on $D^1(E)$ by the formula
\begin{equation}\label{eq_etaT}
  T_\eta(\omega)
  := %\int_\Theta \left(\int_\theta \omega\right) \,  d \eta(\theta) =
\int_\Theta \ld\theta\rd(\omega)\, d \eta(\theta)
\end{equation}
%\end{definition}
for an $\omega\in D^1(E)$.
The following theorem shows that $T_\eta$ is a normal current
under natural assumptions on $\eta$.

%%%%%%%%%%%%%%%%%%%%%%%%%%%%%%%%%%%%%%%%%%%%%%%%%%
\begin{theorem}\label{th_Teta1}
%%%%%%%%%%%%%%%%%%%%%%%%%%%%%%%%%%%%%%%%%%%%%%%%%%
Let $\eta$ be a  transport %finite positice Borel measure on $\Theta$
satisfying
\[
\int_\Theta \MM(\ld\theta\rd)\, d\eta(\theta)<+\infty.
\]
Then~\eqref{eq_etaT}
defines a
normal one-dimensional current $T=T_\eta$ with
\[
  \partial T
  = \eta(1)-\eta(0),\text{ where }
\eta(i):= (e_i)_\#\eta, \quad e_i(\theta):=\theta(i),\qquad i=0,1.
\]
In particular, if $\eta(1)\wedge \eta(0)=0$, then
\[
(\partial T)^+=\eta(1),\qquad (\partial T)^-=\eta(0),
\]
where
$(\partial T)^\pm$ are the positive and the negative part of the measure
$\partial T$ respectively.
Furthermore, 
\[
\MM(T)
\leq \int_\Theta \MM(\ld\theta\rd)\,  d \eta(\theta)\leq\int_\Theta \ell(\theta)\,  d \eta(\theta),
\]
and
for all Borel sets $e\subset E$ one has
\begin{equation}\label{eq_MuTless}
\mu_T(e) %:= M^1(T\res e)
\leq
\int_\Theta \mu_{\ld\theta\rd}( e)\, d\eta(\theta).
\end{equation}
\end{theorem}

\begin{proof}
It suffices to prove that $T=T_\eta$
has finite mass
and finite boundary mass.
According to the definition of mass
\[
|T(f\,d\pi)| \leq \int_{\Theta} \left(\Lip(\pi) \int_E |f|\,d\mu_{\ld\theta\rd}\right)\,
d \eta(\theta),
\]
 which gives~\eqref{eq_MuTless} and hence in particular shows that $T$ has finite mass
\[
  \MM(T) \le  \int_{\Theta} \MM(\ld\theta\rd)\, d \eta(\theta) <+\infty.
\]
Finally, the calculation
\begin{align*}
\partial T(f)
  &= T (1\,df)
  = \int_\Theta \left(\int_0^1 \,d f(\theta(t))
   \right)\, d\eta(\theta)
 =\int_\Theta [f(\theta(1))-f(\theta(0))] \, d\eta(\theta)\\
 &=\int_\Theta f(e_1(\theta))\, d\eta(\theta)-
 \int_\Theta f(e_0(\theta))\, d\eta(\theta)= \int_{E} f(x)\, d (\eta(1)-\eta(0))
\end{align*}
concludes the proof.
\end{proof}

It is worth mentioning that the inequality in~\eqref{eq_MuTless}
may be strict, as the following example shows.

\begin{example}\label{ex_MuT}
Let $\bar e_i$, $i=1,2$ stand for the unit vectors along axis
$x_i$ in $\R^2$, and
let $\Theta_1\subset \Theta$ be a set of paths
$\theta$
in $Q:=[0,1]\times[0,1]$
admitting a parameterization
$\theta(t)=(t, x_2)$, $t\in [0,1]$,
for some $x_2\in [0,1]$.
Define $\eta_1$ by the formula
\[
\eta_1(e):=\H(e_0(e\cap\Theta_1))
\]
for all Borel $e\subset \Theta$, where $e_0(\theta):=\theta(0)$.
Clearly, $T_{\eta_1}=\bar e_1\wedge {\mathcal L}^2\res Q$.
Analogously, letting $\Theta_2\subset \Theta$ be a set of paths
$\theta$ admitting a parameterization
$\theta(t)=(x_1,t)$, $t\in [0,1]$,
for some $x_1\in [0,1]$, and
defining $\eta_2$ by the formula
\[
\eta_2(e):=\H(e_0(e\cap\Theta_2))
\]
for all Borel $e\subset \Theta$, we get
$T_{\eta_2}=\bar e_2\wedge {\mathcal L}^2\res Q$.
Now, setting $\eta:=\eta_1+\eta_2$, one
has
$T_\eta = T_{\eta_1}+T_{\eta_2}=(\bar e_1+\bar e_2) \wedge{\mathcal L}^2\res Q$, and
hence, $\MM(T_\eta)=\sqrt{2}$, while
\[
\int_\Theta \MM(\ld\theta\rd)\, d\eta=
\int_{\Theta_1} \MM(\ld\theta\rd)\, d\eta_1 +
\int_{\Theta_2} \MM(\ld\theta\rd)\, d\eta_2 =
2 > \MM(T_\eta).
\]
\end{example}

We now consider a converse statement, i.e.\
when for a
given
normal current $T\in \M_1(E)$, there is a transport $\eta$
satisfying $T=T_\eta$. For this purpose we give the following definition.

\begin{definition}\label{def_SmirnDecompAc1}
We say that a
normal  current $T\in \M_1(E)$ is decomposable in curves, if
there is a transport $\eta$
satisfying $T=T_\eta + C$, where $C\leq T$ is
a cycle of $T$, $T_\eta\leq T$, and the equalities
\begin{equation}\label{eq:claim1a}
\MM(T_\eta)
= \int_\Theta \MM(\ld\theta\rd)\,  d \eta(\theta)=\int_\Theta \ell(\theta)\,  d \eta(\theta),
\end{equation}
and
\begin{equation}
  \label{eq:claim2}
  \eta(1) = (\partial T)^+,\quad \eta(0) = (\partial T)^-
\end{equation}
% where
%$(\partial T)^\pm$ are the positive and the negative part of the measure
%$\partial T$ respectively,
are valid.
\end{definition}

\begin{remark}\label{rem_claim1a}
  In view of Theorem~\ref{th_Teta1}, the claim~\eqref{eq:claim1a}
is equivalent to a formally weaker one
\[
%\begin{equation}\label{eq:claim1b}
\MM(T_\eta)
\geq 
%\int_\Theta \MM(\ld\theta\rd)\,  d \eta(\theta).
\int_\Theta \ell(\theta)\,  d \eta(\theta).
%\end{equation}
\]
\end{remark}

%\begin{remark}\label{rem_claim2}
Note that the property of being decomposable in curves for acyclic currents (i.e.\ with $C=0$) is exactly what is claimed
in theorem~C from~\cite{Smirnov94} for classical Whitney one-dimensional normal acyclic currents in a Euclidean space ($E=\R^n$);
if one decides to be meticulous, one has to mention also that there, instead of claim~\eqref{eq:claim2},
a different (though equivalent) claim
\[
\mu_{\partial T} = \int_{\Theta(E)} (\delta_{\theta(0)}+ \delta_{\theta(1)})\, d\eta(\theta)
\]
is formulated. In Theorem~\ref{th_decompNorm1acycl} we will show this property for all one-dimensional Ambrosio-Kirchheim
normal acyclic currents in any metric space (up to the set-theoretic assumption made in the Introduction), thus generalizing
the mentioned result from~\cite{Smirnov94} to metric currents.
%
%For acyclic currents, the property of being decomposable in curves
%as introduced in the above definition is,
%apart from the claim~\eqref{eq:claim2}, exactly the property
% stated (though in quite different
%terminology) in theorem~C from~\cite{Smirnov94} for all normal acyclic currents in a Euclidean space.
%\end{remark}

We now are able to prove the following statement which is
%one of the principal results
the principal tool
of this paper.

%%%%%%%%%%%%%%%%%%%%%%%%%%%%%%%%%%%%%%%%%%%%%%%%%%
\begin{proposition}\label{prop_Teta2}
%%%%%%%%%%%%%%%%%%%%%%%%%%%%%%%%%%%%%%%%%%%%%%%%%%
Let $T\in \M_1(E)$ be an
acyclic
normal current
such that there is a sequence of normal currents
$T_\nu\in \M_1(E)$ decomposable in curves
with $T_\nu\weakto T$ weakly in the sense of currents,
$(\partial T_\nu)^\pm\weakto (\partial T)^\pm$ in the narrow sense of measures,
and $\MM(T_\nu)\to \MM(T)$ as $\nu\to \infty$. %, while the measures $\mu_{T_\nu}$ are uniformly tight.
Then $T$ is decomposable in curves, and in particular $T=T_\eta$ for
some transport $\eta$, and
$\eta$-a.e.\ $\theta\in \Theta$ is an arc.
\end{proposition}
%%%%%%%%%%%%%%%%%%%%%%

\begin{remark}\label{rem_Teta2a1}
If $T\in \M_1(E)$ is an
acyclic
normal current decomposable in curves, then applying the above theorem
for $T_\nu:=T$ we get
that in particular $T=T_\eta$
for
some transport $\eta$ such that
relationships~\eqref{eq:claim1a} and~\eqref{eq:claim2} hold, and
$\eta$-a.e.\ $\theta\in \Theta$ is an arc.

Thus,
a generic (not necessarily acyclic) normal current $T\in \M_1(E)$  is decomposable in curves, if and only if
there is a transport $\eta$
satisfying all the properties of Definition~\ref{def_SmirnDecompAc1}
(i.e.\ $T=T_\eta + C$, where $T_\eta\leq T$ is acyclic and $C\leq T$ is
a cycle of $T$, the
relationships~\eqref{eq:claim1a} and~\eqref{eq:claim2} holding true), with the additional property
that
$\eta$-a.e.\ $\theta\in \Theta$ is an arc.
\end{remark}

\begin{remark}\label{rem_gentight1}
In the proof of the above Proposition~\ref{prop_Teta2} (in particular, in the key auxiliary
instrument, Proposition~\ref{prop_narrconv}) we heavily rely on the fact that the measures
$\mu_{T_\nu}$ as well as $(\partial T_\nu)^\pm$ are uniformly tight. For $\mu_{T_\nu}$ this is true
in view of Lemma~\ref{lm_AKtight1}, while for $(\partial T_\nu)^\pm$ directly from
the Prokhorov theorem
for nonnegative measures
(theorem~8.6.4 from~\cite{Bogachev06}). However, in both arguments one silently admits the set-theoretic assumption
made in the Introduction; without the latter one has to assume that each of the measures
$\mu_{T_\nu}$, $(\partial T_\nu)^\pm$, $\mu_{T}$ and $(\partial T)^\pm$ is tight itself. Under such an assumption
the statement remains true minding Remark~\ref{rem_gentight3} (for uniform tightness of
$\mu_{T_\nu}$)  as well as theorem~8.6.4 from~\cite{Bogachev06} (for uniform tightness of $(\partial T_\nu)^\pm$).
\end{remark}
%%%%%%%%%%%%%%%%%%%%%%

\begin{proof}%[Proof of Proposition~\ref{prop_Teta2}]
%We divide the proof in two steps.
%
%\emph{Step 1.}
We may assume
$\MM(T_\nu)\le C$ (one can take e.g. $C:=\MM(T)+1$).
Decomposability of $T_\nu$ in curves means the existence
for each $T_\nu$ of a transport $\eta_\nu$ such that
\begin{equation}\label{eq_Tom_12}
\begin{aligned}
T_\nu (f\,d\pi)&=T_{\eta_\nu} (f\,d\pi)+ T_\nu'(f\,d\pi), \qquad T_\nu'\leq T_\nu,\qquad \partial T_\nu'=0,\\
T_{\eta_\nu} (f\,d\pi)
  &= \int_\Theta \ld\theta\rd(f\,d\pi)\,  d \eta_\nu(\theta),\\
\MM(T_{\eta_\nu})&=\int_\Theta \MM(\ld\theta\rd)\,
d\eta_\nu(\theta)=\int_\Theta \ell(\theta)\, d\eta_\nu(\theta),\\
\eta_\nu(1) &=(\partial T_{\eta_\nu})^+, \qquad \eta_\nu(0) =(\partial T_{\eta_\nu})^-
\end{aligned}
\end{equation}
for all $f\,d\pi\in D^1(E)$
(in particular, by any of the last two equalities, the total masses $\eta_\nu(\Theta)$ are uniformly bounded).
Since $\MM(T_\nu')\leq\MM(T_\nu)\leq C$, $\partial T_\nu'=0$, and
$\mu_{T_\nu'}\leq \mu_{T_\nu}$ by Remark~\ref{rem_subcurr4},
the latter measures being uniformly tight,
by Lemma~\ref{lm_AKtight1},
 hence so being also the former,
then by compactness theorem~5.2 from~\cite{AmbrKirch00} one has $T_\nu'\weakto T'$ with
$T'\leq T$ and $\partial T'=0$ by Lemma~\ref{lm:compconv}.
Since $T$ is acyclic, then $T'=0$, and hence
$T_{\eta_\nu}\weakto T$.

In view of~\eqref{eq_Tom_12} we have  the estimate
\[
\int_\Theta \ell(\theta)\,d\eta_\nu = \MM(T_{\eta_\nu})\leq C.
\]
We may invoke therefore Proposition~\ref{prop_narrconv} below, obtaining that
up to a subsequence (not relabeled)
$\eta_\nu\weakto\eta$ in the narrow sense of measures
for some finite Borel measure $\eta$, and, moreover,
that one may pass to the limit as $\nu\to \infty$
in both sides of the first
relationship of~\eqref{eq_Tom_12} obtaining therefore
$T(f\,d\pi)=T_\eta(f\,d\pi)$ for each $f\,d\pi\in D^1(E)$, and
hence $T=T_\eta$.
One shows in addition
that~\eqref{eq:claim2}
is valid by passing to the limit  as $\nu\to \infty$
in both sides of the last two equalities from~\eqref{eq_Tom_12}.

Furthermore, note that
\begin{equation}\label{eq_Tom4}
\MM(T_{\eta_\nu}) = \int_\Theta \MM(\ld\theta\rd)\,  d \eta_\nu(\theta)
\end{equation}
by the second relationship of~\eqref{eq_Tom_12}.
Hence, minding that the functional
$\theta\in \Theta\mapsto \MM(\ld\theta\rd)$ is l.s.c., and hence, the integral
in the right-hand side of the above relationship is l.s.c. with respect to
narrow convergence of $\eta_\nu$, by
passing to a limit in both sides of~\eqref{eq_Tom4}
as $\nu\to \infty$,
we deduce
\[
%\begin{aligned}
\MM(T)=\lim_\nu \MM(T_{\eta_\nu})
%&
=  \lim_\nu
\int_\Theta \MM(\ld\theta\rd)\,  d \eta_\nu(\theta)
%\\
%&
\geq  \int_\Theta \MM(\ld\theta\rd)\,  d \eta(\theta),
%\end{aligned}
\]
which provides~\eqref{eq:claim1a} once one recalls
Remark~\ref{rem_claim1a}.

Consider %for further use
also the functional defined over transports by
%$M'\colon \Theta\to \R^+$
%defined by
%\begin{equation}\label{eq_Leta}
\[
%M'(\eta):=
\eta\mapsto \int_\Theta \ell(\theta)\, d\eta.
\]
%\end{equation}
It
is l.s.c. with respect to the narrow convergence of measures (because
the parametric length $\ell(\cdot)$ is l.s.c. in $\Theta$).
Hence, minding that for each $\eta_\nu$ one has %by construction
\[
\MM(T_\nu)=\int_\Theta \ell(\theta)\, d\eta_\nu=\MM(T_\nu),
\]
we get
\[
\int_\Theta \ell(\theta)\, d\eta\leq \MM(T).
\]
Minding that the opposite inequality holds in view of Theorem~\ref{th_Teta1}, we get that in fact the above inequality is the equality and in particular,
we have that $\MM(\ld\theta\rd)=
\ell(\theta)$ for $\eta$-a.e.\ $\theta\in \Theta$.

%%%%%
Let
$f$: $\Theta\to \Theta$ and
$g$: $\Theta\to \Theta$ be given by Lemma~\ref{lm_delcycle1}.
Then, minding
\begin{align*}
\MM(\ld g(\theta)\rd)\leq\ell(g(\theta)), \qquad\qquad
\MM(\ld f(\theta)\rd)\leq\ell(f(\theta)),
\end{align*}
we get
\[
\MM(\ld g(\theta)\rd)+\MM(\ld f(\theta)\rd)\leq\ell(g(\theta))+\ell(f(\theta))=\ell(\theta)=\MM(\ld \theta\rd),
\]
hence $\MM(\ld g(\theta)\rd)+\MM(\ld f(\theta)\rd)=\MM(\ld \theta\rd)$
for $\eta$-a.e.\ $\theta\in \Theta$.
By virtue of this one has that $T_{f_\#\eta}\leq T$ and therefore $T_{f_\#\eta}$ is a cycle of $T=T_\eta$.
Hence, $T_{f_\#\eta}=0$, so that $T_{g_\#\eta}=T_\eta=T$.
 This also means $\ld f(\theta)\rd =0$ for $\eta$-a.e.\
$\theta\in \Theta$.
For such $\theta$ we have thus $\ld g(\theta)\rd = \ld \theta\rd$,
and in particular, the chain of inequalities
\[
\ell(\theta)=\MM(\ld \theta\rd)=\MM(\ld g(\theta)\rd)\leq \ell(g(\theta))\leq \ell(\theta)
\]
is true, which gives
$\ell(g(\theta)) =\ell(\theta)$, hence $\eta$-a.e.\ $\theta \in
\Theta$ is an arc as claimed.
\end{proof}

The statement below is the key technical result used in the proof of Proposition~\ref{prop_Teta2}.

\begin{proposition}\label{prop_narrconv}
Let
$\{\eta_\nu\}$ be a sequence of nonnegative
finite Borel measures over $\Theta$
with uniformly bounded total masses, and denote $T_\nu:= T_{\eta_\nu}$.
Assume that
$T_\nu\weakto T$ weakly in the sense of currents,
$\MM(T_\nu)\to \MM(T)$
as $\nu\to \infty$,
%the measures $\mu_{T_\nu}$ being uniformly tight,
and
\begin{equation}\label{eq_MTnu1}
\MM(T_\nu)= \int_\Theta \ell(\theta)\, d\eta_\nu \leq C <+\infty
\end{equation}
for all $\nu\in \N$, while the current $T$ is acyclic.
 Then
there exists a transport $\eta$ such that
 up to a subsequence (not relabeled),
$\eta_\nu\weakto \eta$ (and in particular,
$\eta_\nu(i)\weakto \eta(i)$, $i=0,1$)
in the narrow sense of measures, while $T=T_\eta$.
\end{proposition}

\begin{remark}\label{rem_narrconv1}
If in the statement of Proposition~\ref{prop_narrconv}
we required that all $\eta_\nu$ be concentrated on some compact subset of $C\subset \Theta$ of curves with uniformly bounded lengths,
then the assumption of acyclicity of the limit current $T$ is unnecessary
and the proof is quite immediate. In fact, in this case
one has that, up to a subsequence (not relabeled),
$\eta_\nu \weakto \eta$ as $\nu\to\infty$
in the $*$-weak sense of measures over $C$
for some finite Borel measure $\eta$ over $C$.
Then
one immediately gets
\[
T_\nu(\omega)= \int_\Theta\ld\theta\rd(\omega)\, d\eta_\nu(\theta)
\to \int_\Theta
\ld\theta\rd(\omega)\, d\eta(\theta) =T_\eta(\omega)
\]
as $\nu\to \infty$, since the function
$\theta\in C\mapsto \ld\theta\rd(\omega)$ is continuous by Lemma~\ref{lm_Fcontinuous}.
Hence $T=T_\eta$.
The convergence $\eta_\nu(i)\weakto\eta(i)$, $i=0,1$, as $\nu\to\infty$ follows from
the fact that a push-forward operator
by means of a continuous function is continuous with respect to
$*$-weak convergence of measures.

Thus, the main difficulty in proving Proposition~\ref{prop_narrconv}
is that we cannot say a priori that $\eta_\nu$
are concentrated in some compact subset of $\Theta$. In this case acyclicity of the limit current will be important as we will see in Example~\ref{ex_convcycle} below.
\end{remark}

\begin{proof}
Combine Corollary~\ref{co_narrconv1b} and Lemma~\ref{lm_narrconv1c} below.
\end{proof}

The results below are used in the proof of Proposition~\ref{prop_narrconv}.

\begin{lemma}\label{lm_narrconv1a}
Under the conditions of Proposition~\ref{prop_narrconv} there is an
increasing sequence of compact sets $\bar\Theta_k'\subset \Theta$ such that
$\eta_\nu((\bar\Theta_k')^c)\leq C/2^k$ for some $C>0$ independent of $k$ and $\nu$
and for all $\nu\in \N$. This is true without any assumption on the acyclicity of the limit current $T$.
\end{lemma}

\begin{proof}
By %Proposition~\ref{prop_Prokhorov_metrcurr1}
uniform tightness of $\mu_{T_\nu}$ there is a sequence
of compact sets $K_k\subset E$ such that
$\mu_{T_\nu}(K_k^c) \leq 1/ 4^k$.
Consider the set
\[
\Theta_k:=\{\theta\in \Theta\,:\,
%\ell(\theta)=\MM(\ld\theta\rd)\mbox{ and }
\mu_{\ld\theta\rd} (K_k^c) > 1/2^k\}.
\]
One has then
\begin{align*}
\frac 1 {4^k} \geq \mu_{T_\nu}(K_k^c)& =\int_\Theta \mu_{\ld\theta\rd} (K_k^c) \, d\eta_\nu(\theta)\\
& \geq \int_{\Theta_k} \mu_{\ld\theta\rd} (K_k^c) \, d\eta_\nu(\theta)
> \frac 1 {2^k} \eta_\nu(\Theta_k),
\end{align*}
so that
$\eta_\nu(\Theta_k)\leq 1/2^k$. Letting then
\[
\hat \Theta_j:= \bigcap_{k\geq j}\Theta_k^c,
\]
we get
\begin{equation}\label{eq_Fthetp1}
\eta_\nu(\hat\Theta_j^c)= \eta_\nu\left(\bigcup_{k\geq j}\Theta_k\right)\leq \sum_{k=j}^\infty \eta_\nu(\Theta_k)\leq \frac 1 {2^{j-1}}.
\end{equation}

We also observe that for
\[
\Theta^j:=\{\theta\in\Theta\,:\, \ell(\theta)> 2^j\}
\]
one has
\begin{align*}
C\geq \MM(T_\nu)=
\int_\Theta \ell(\theta)\,d\eta_\nu &\geq
\int_{\Theta^j} \ell(\theta)\,d\eta_\nu > 2^j\eta_\nu(\Theta^j),
\end{align*}
hence $\eta_\nu(\Theta^j) < C/2^j$.
Finally, minding that the measures $\mu_{\partial T_\nu}$
are also uniformly tight, we get the existence,
of a sequence of compact sets $\tilde K_k\subset E$ such that
for each $j\in \N$ and for
\[
\tilde\Theta_j:= \{\theta\in \Theta\,:\, \theta(0)\in \tilde K_j\mbox{ or } \theta(1)\in \tilde K_j\}
\]
one has $\eta_\nu(\tilde\Theta_j^c)\leq 1/2^j$ for all $\nu\in \N$.

Let then
\[
\Theta'_j := (\hat\Theta_j\setminus \Theta^j)\cap \tilde\Theta_j\cap \{\theta\in \Theta\,:\,
\MM(\ld\theta\rd)=\ell(\theta)\}.
\]
Note that
for $\eta_\nu$-a.e.\ $\theta$ one has $\MM(\ld\theta\rd)=\ell(\theta)$
(this is true in view of~\eqref{eq_MTnu1} minding~\eqref{eq_MuTless} and the inequality $\MM(\ld\theta\rd)\leq \ell(\theta)$).
Recalling then~\eqref{eq_Fthetp1},
we arrive at the estimate
\begin{equation}\label{eq_Fthet2}
\begin{aligned}
\eta_\nu((\Theta_j')^c) = \eta_\nu((\hat\Theta_j\setminus \Theta^j)^c\cup \tilde\Theta_j^c) & =
\eta_\nu(\hat\Theta_j^c\cup \Theta^j \cup \tilde\Theta_j^c) \leq \frac C {2^j},
\end{aligned}
\end{equation}
and in particular, all $\eta_\nu$ are concentrated on
\[
\Theta':=\bigcup_j \bar \Theta_j'.
\]

Observe now that each $\Theta_j'$ is a precompact subset of $\Theta$. In fact, every $\theta\in \Theta_j'$ with constant velocity over $[0,1]$ satisfies $\Lip\, \theta\leq 2^j$, while
\begin{align*}
\mu_{\ld\theta\rd}(K_k^c)\leq \int_{\theta^{-1}(K_k^c)}|\dot{\theta}(t)|\, dt &= \ell(\theta)
\mathcal{L}^1(\theta^{-1}(K_k^c)),\\
\mu_{\ld\theta\rd}(K_k)\leq\int_{\theta^{-1}(K_k)}|\dot{\theta}(t)|\, dt &= \ell(\theta)
\mathcal{L}^1(\theta^{-1}(K_k)),\\
\mu_{\ld\theta\rd}(E)=\int_0^1|\dot{\theta}(t)|\, dt &= \ell(\theta),
\end{align*}
hence by summing the above two inequalities and comparing them to the third equality, we have that in fact the equalities hold, and thus in particular,
\[
\mathcal{L}^1(\theta^{-1}(K_k^c))= \mu_{\ld\theta\rd}(K_k^c)/\ell(\theta).
\]
Thus, for a sequence $\theta_\nu\in \Theta_j'$ one has that either
\begin{itemize}
\item there is
a subsequence (not relabeled) such that $\ell(\theta_\nu)\geq c>0$, which
implies
\[
\mathcal{L}^1(\theta^{-1}(K_k^c))\leq\frac{1}{2^k c},
\]
for each $k\geq j$,
while $\Lip\,\theta_\nu\leq \ell(\theta_\nu)\leq 2^j$, so that
this subsequence is compact in $\Theta$ by Proposition~\ref{prop_AscoliArzela_metr1}, or
\item $\ell(\theta_\nu)\to 0$ as $\nu\to \infty$. In this case one has that either $\theta_\nu(0)\in \tilde K_j$ or $\theta_\nu(1)\in \tilde K_j$ for a
     subsequence of $\nu\in \N$. Since both cases are symmetric, we assume $\theta_\nu(0)\in \tilde K_j$. One has up to a subsequence (not relabeled), $\theta_\nu(0)\to x\in \tilde K_j$. Then for every $z_\nu\in \theta_\nu$ due to the estimate
    \[
    d(z_\nu, \theta_\nu(0))\leq \ell(\theta_\nu)
    \]
    we get $z_\nu\to x$, so that $\theta_\nu$ converges in $\Theta$ to a constant curve concentrated on $x\in E$.
\end{itemize}
This concludes the proof of the Lemma.
\end{proof}

\begin{corollary}\label{co_narrconv1b}
With the notation of Lemma~\ref{lm_narrconv1a} under the conditions of Proposition~\ref{prop_narrconv}
all $\eta_\nu$ are concentrated over the set $\Theta':=\cup_j\bar \Theta_j'$. Denoting by $\bar\Theta'$ the
closure of the latter in the space of continuous
functions $C([0,1];E)$ (factorized by parameterization) with the uniform metric, we have that up to a subsequence (not relabeled),
$\eta_\nu \weakto \eta$ and $\eta_\nu(i)\weakto\eta(i)$, $i=0,1$, as $\nu\to\infty$
in the narrow sense of measures
for some finite Borel measure $\eta$ over $\bar\Theta'$ concentrated over $\Theta'$.
This is again true without any assumption on the acyclicity of the limit current $T$.
\end{corollary}

\begin{proof}
Clearly $\bar\Theta'$ defined in the statement being proven is a Polish space (since already $\Theta'$ is $\sigma$-compact, hence separable), and since
the sequence
$\eta_\nu$ is uniformly tight on $\bar\Theta'$ and has uniformly bounded mass,
then, up to a subsequence (not relabeled),
$\eta_\nu \weakto \eta$ as $\nu\to\infty$
in the narrow sense of measures
for some finite Borel measure $\eta$ over $\bar\Theta'$.
Note that in view of Lemma~\ref{lm_narrconv1a} one has
$\eta_\nu((\bar\Theta_j')^c)\leq C/2^j$ for all $\nu\in \N$ (here and below the complement is meant now with respect to $\bar\Theta'$),
hence $\eta((\bar\Theta_j')^c)\leq C/2^j$ and therefore $\eta$ is concentrated on $\Theta'$.
The convergence $\eta_\nu(i)\weakto\eta(i)$, $i=0,1$, as $\nu\to\infty$ again follows from
continuity with respect to
narrow convergence of measures of the push-forward operator
by means of a continuous function.
\end{proof}

\begin{lemma}\label{lm_narrconv1c}
Under the conditions of Proposition~\ref{prop_narrconv},
let $\eta$ be a limit point  in the narrow topology of $\eta_\nu$ pointed out in Corollary~\ref{co_narrconv1b}. Then $T = T_\eta$.
\end{lemma}

\begin{proof}
We use the notation of Lemma~\ref{lm_narrconv1a} and Corollary~\ref{co_narrconv1b}.
Let us show first that
\begin{equation}\label{eq_equiint1}
\phi(k):=  \limsup_\nu \int_{C_k^c}\ell(\theta)\, d\eta_\nu(\theta)
\to 0 \mbox{ when } k\to \infty,
\end{equation}
where
$C_k:=\{\theta\in \bar\Theta'\,:\, \ell(\theta)\leq k\}$.
It is here that the assumption on the acyclicity of $T$ enters in play.
To prove~\eqref{eq_equiint1} assume the contrary.
Then there exists a $c>0$ such that for a subsequence
of $\eta_\nu$ (not relabeled) one has
\[
\int_{\{\ell(\theta)>\nu\}}\ell(\theta)\, d\eta_\nu(\theta)\geq c.
\]
Consider then $\eta_\nu':=\eta_\nu\res \{\ell(\theta)>\nu\}$, and
$S_\nu:= T_{\eta_\nu'}$. By Lemma~\ref{lm_PS7etasub1},
each $S_\nu$ is a subcurrent of $T_\nu$, 
with $\partial S_\nu\weakto 0$ (see below), % una precisazione: 17.09.13
and hence
by Lemma~\ref{lm:compconv} one gets that up to a subsequence
(again not relabeled) $S_\nu\weakto S$ weakly in the sense of currents as $\nu\to\infty$,
while $S$ is a subcurrent of $T$ and $\MM(S)\geq c$.
On the other hand,
$\eta_\nu'\weakto 0$ (in fact, even the total masses of $\eta_\nu'$ converge to zero)
because
by~\eqref{eq_MTnu1} one has
\[
\nu\eta'_\nu(\Theta)=
\nu \int_{\{\ell(\theta)>\nu\}} \, d\eta_\nu \leq
\int_{\{\ell(\theta)>\nu\}} \ell(\theta)\, d\eta_\nu\leq
\int_\Theta \ell(\theta)\, d\eta_\nu \leq C <+\infty.
\]
%by Lemma~\ref{lm_convzero} (applied with $\eta_\nu'$ in place of $\eta_\nu$)
Thus
\[
\partial S_\nu=\eta_\nu'(1)-\eta_\nu'(0)\weakto 0
\]
weakly in the sense of measures as $\nu\to\infty$ (again, in fact also in mass although we do not really need it),
hence $\partial S=0$ and, by
acyclicity of $T$,
one gets
$S=0$, giving a contradiction. Hence, the claim~\eqref{eq_equiint1}
is proven.

Fix now an arbitrary
$\omega:=f\,d\pi\in D^1(E)$ with $f\in L^\infty$, denoting for the sake
of brevity $|\omega|:= \|f\|_{\infty}\Lip\,\pi$.
Mind that
\[
|\ld\theta\rd(\omega)|\leq |\omega| k
\]
whenever $\theta\in C_k$.
For each $k\in \N$
using the classical Tietze-Urysohn extension theorem we can find a
bounded continuous function $f_k\colon \bar\Theta'\to \R$
satisfying
\begin{align*}
f_k(\theta) =
\ld\theta\rd(\omega),&\qquad\qquad  \mbox{ if }\theta\in C_k,\\
|f_k(\theta)|  \leq |\omega| k & \qquad\qquad  \mbox{ for all }\theta\in \bar\Theta'.
\end{align*}
We have now
\begin{equation}\label{eq_thetfk_conv1}
    \int_{\bar\Theta'} f_k(\theta)\, d\eta_\nu(\theta) \to
    \int_{\bar\Theta'} f_k(\theta)\, d\eta(\theta)
\end{equation}
as $\nu\to \infty$.
On the other hand,
\[
\int_{\bar\Theta'} f_k(\theta)\, d\eta_\nu(\theta) - \int_{\bar \Theta'} \ld\theta\rd(\omega)\, d\eta_\nu(\theta)=
\int_{C_k^c} \left(f_k(\theta)-\ld\theta\rd(\omega)\right) \, d\eta_\nu(\theta),
\]
while
\begin{equation}\label{eq_thetfk_conv2}
\begin{aligned}
\left|\int_{C_k^c} \left(f_k(\theta)-\ld\theta\rd(\omega)\right) \, d\eta_\nu(\theta) \right| &\leq
\int_{C_k^c}
\left| f_k(\theta)\right|
 \, d\eta_\nu(\theta) + \int_{C_k^c}
\left| \ld\theta\rd(\omega)\right|
 \, d\eta_\nu(\theta) \\
 &\leq
\int_{C_k^c} |\omega|k\, d\eta_\nu(\theta)
+\int_{C_k^c} |\omega|\ell(\theta)\, d\eta_\nu(\theta)
\\
&\leq 2|\omega| \int_{C_k^c} \ell(\theta)\, d\eta_\nu(\theta)\leq 4|\omega|\phi(k)
\end{aligned}
\end{equation}
for all sufficiently large $\nu\in \N$.
Analogously,
\begin{equation}\label{eq_thetfk_conv3}
\int_{\bar\Theta'} f_k(\theta)\, d\eta(\theta) - \int_{\bar \Theta'} \ld\theta\rd(\omega)\, d\eta(\theta)=
\int_{C_k^c} \left(f_k(\theta)-\ld\theta\rd(\omega)\right) \, d\eta(\theta),
\end{equation}
with
\begin{align*}
\left|\int_{C_k^c} \left(f_k(\theta)-\ld\theta\rd(\omega)\right) \, d\eta(\theta) \right|  &\leq 2|\omega|\phi(k),
\end{align*}
where we used the fact that since the set $C_k$ is closed in $\bar\Theta'$, then by~\eqref{eq_equiint1} one has
\[
\int_{C_k^c}\ell(\theta)\, d\eta(\theta)\leq \phi(k).
\]

Thus, in view of~\eqref{eq_thetfk_conv1} and~\eqref{eq_thetfk_conv2}
we get
\begin{align*}
\int_\Theta f_k(\theta)\, d\eta(\theta)-4 |\omega| \phi(k) &\leq
\liminf_\nu \int_\Theta \ld\theta\rd(\omega)\, d\eta_\nu(\theta)
\leq
\limsup_\nu \int_\Theta \ld\theta\rd(\omega)\, d\eta_\nu(\theta)\\
&\leq
\int_\Theta
f_k(\theta)\, d\eta(\theta)+4|\omega| \phi(k).
\end{align*}
Minding~\eqref{eq_thetfk_conv3}, we arrive at the estimate
\begin{align*}
\int_\Theta \ld\theta\rd(\omega)\, d\eta(\theta)-6 |\omega| \phi(k) &\leq
\liminf_\nu \int_\Theta \ld\theta\rd(\omega)\, d\eta_\nu(\theta)
\leq
\limsup_\nu \int_\Theta \ld\theta\rd(\omega)\, d\eta_\nu(\theta)\\
&\leq
\int_\Theta
\ld\theta\rd(\omega)\, d\eta(\theta)+6|\omega| \phi(k),
\end{align*}
and passing to the limit as $k\to \infty$, we get
\[
\lim_\nu T_\nu(\omega)=\int_\Theta
\ld\theta\rd(\omega)\, d\eta(\theta) =T_\eta(\omega)
\]
as $\nu\to \infty$, which
allows us to conclude that $T=T_\eta$.
\end{proof}

It is worth remarking that the requirement of acyclicity of the
limit current $T$ of the above Proposition~\ref{prop_narrconv}
is essential as shown in the example below.

\begin{example}\label{ex_convcycle}
 Consider the sequence of
 curves in $\R^2$ admitting the parameterization
$\theta_\nu(t):=(1+t/\nu)(\cos(2\pi \nu t),\sin(2\pi \nu t))$,
$t\in[0,1]$, and define $\eta_\nu:=\frac 1 \nu
\delta_{\theta_\nu}$ be the transport concentrated on
$\theta_\nu\in\Theta$ and having total mass $1/\nu$.
Define also $\bar \theta(t):=(\cos(2\pi t),\sin(2\pi t))$ and let
$\eta:=\delta_{\bar{\theta}}$ be the transport concentrated
on $\bar \theta$ with unit total mass.
Clearly $\eta_\nu\weakto 0$ in the narrow sense of measures as
$\nu\to\infty$ (in fact, $\eta_\nu(\Theta)=1/\nu$). On the other hand,
$T_{\eta_\nu}\weakto T_\eta\neq 0$ as $\nu\to\infty$ .
However, this is not in contradiction with the above Proposition~\ref{prop_narrconv}
because clearly $\partial T_\eta=0$, i.e.\ $T_\eta$ is a cycle.
\end{example}

Another lemma used in the proof of Proposition~\ref{prop_Teta2} is provided below.

%andava bene anche com'era, con f(\theta) loop of maximum parametric length %in \theta
\begin{lemma}\label{lm_delcycle1}
The following assertions are valid.
\begin{itemize}
\item[(i)] There is a map $f$: $\Theta\to \Theta$
measurable with respect to all transports
such that
$f(\theta)$ is a loop  (i.e.\ a %simple 
closed curve)
%of maximum parametric length
contained in
$\theta\in\Theta$ %,i.e.\
with
\[
\ell(f(\theta))\geq 1/2 \sup\left\{\ell(\sigma)\,:\, \sigma \mbox{
is a loop contained in } \theta \right\}.
\]
%(the latter supremum being attained).
\item[(ii)] There is a  map $g$: $\Theta\to \Theta$
measurable with respect to all transports
such that
for all $\theta\in \Theta$ one has %that
$\theta=g(\theta)\cup f(\theta)$ (as traces),
$\ld \theta\rd = \ld g(\theta)\rd +\ld f(\theta)\rd $, %while
\[
\ell(g(\theta))< \ell(\theta),
%\qquad\qquad
%\H(g(\theta))< \H(\theta),
\]
unless $\theta$ is an arc, and, finally,
$g(\theta)=\theta$, if and only if $\theta$ is an arc.
\end{itemize}
\end{lemma}

\begin{proof}
We construct a map $f$: $\Theta\to \Theta$ satisfying claim~(i) as
follows. For every $\theta\in \Theta$ and $x\in \theta$ we let
$C(\theta,x)$ stand for the set of curves contained in $\theta$
starting and ending at $x$ in the sense that
\begin{align*}
C(\theta,x)=\Big\{ \tilde\theta\in \Theta\colon & \tilde\theta(t)=
\theta((1-t)s_1+ts_2)\\
& \mbox{ for some } 0\leq s_1\leq s_2\leq 1,\,
\theta(s_1)=\theta(s_2)=x \Big\}.
\end{align*}
In case $x\not\in \theta$ we define $C(\theta,x)$ to be a set
consisting just of a single curve $\theta_x$ defined by
$\theta_x(t):=x$ for all $t\in [0,1]$, i.e.\ of a ``constant'' curve
the trace of which reduces to just one point $x$. Note that
$\theta_x\in C(\theta,x)$ for all $x\in E$. Defined in this way,
the multivalued map
\[
(\theta,x)\in \Theta\times\R^n\mapsto C(\theta,x)\subset \Theta
\]
is u.s.c. (as a multivalued map), and hence Borel measurable.
Therefore, recalling that $\ell\colon \Theta\to \R$ is l.s.c. one
gets the Borel measurability of the single-valued map
\[
\lambda:\, \theta\in \Theta\mapsto
\sup_{x\in \R^n} \sup\{\ell(\sigma)\,:\,
\sigma\in C(\theta,x)\}\in \R.
\]
Clearly, $\lambda(\theta)$ gives
%the length of a maximal loop
is the supremum of the length of the loops
 contained in $\theta$. Finally, we define
\[
F:\, \theta\in \Theta \mapsto \left\{\sigma\in \bigcup_{x\in \theta}
C(\theta,x)\,:\,
\ell(\sigma)%=
\geq
 \lambda(\theta)
/2
 \right\}\subset \Theta.
\]
By the von Neumann-Aumann measurable selection theorem~\cite{Srivast98}[corollary~5.5.8]
%(theorem~III.22 and III.23 from~\cite{CasVal}, or, equivalently,
 one can find a selection
$f\colon \Theta\to \Theta$ of the multivalued map $F$ which is
measurable with respect to all transports $\eta$. Clearly,
$f(\theta)$ is as announced in  the statement being proven.
%a loop of maximal length contained in $\theta$.

Define now $g\colon \Theta\to \Theta$
as a union of two curvilinear segments,
by setting
\[
g(\theta):= [\theta(0),f(\theta)(0)]\circ
[f(\theta)(1), \theta(1)].
\]
Clearly, $g(\theta)$ is obtained by ``cancelling'' the loop
$f(\theta)$ from $\theta$.
The properties of $g$ announced in claim~(ii) follow immediately
since
%$\ell(g(\theta))=\ell(\theta)-\lambda(\theta)$,
$\ell(g(\theta))\leq \ell(\theta)-\lambda(\theta)/2$, while
$g(\theta)=\theta$, if and only if $f(\theta)=\theta_x$ for some
$x\in \theta$, i.e.\ when $\theta$ is an arc.
%\qed
\end{proof}

Finally, the following two elementary observations have also to be mentioned because used in the proof, as well as because of having some independent interest.

\begin{lemma}\label{lm_PS7etasub1}
Let $\eta$ be a transport satisfying~\eqref{eq:claim1a} for a
normal current $T=T_\eta\in \M_1(E)$, 
%with compact support, %%% non serve supp comp, residuo della vers precedente; corretto il 17.09.13
and let $\tilde\eta$ be another transport such that $\tilde\eta\leq
\eta$. Then for $\tilde T:= T_{\tilde\eta}$ one has $\tilde T\leq T$
and 
\[\MM(\tilde T) = \int_\Theta \MM(\ld\theta\rd)\,
d\tilde\eta(\theta) = \int_\Theta \ell(\theta)\,
d\tilde\eta(\theta).
\]
\end{lemma}

\begin{proof}
Let $\eta':=\eta-\tilde\eta$ and $T':=T_{\eta'}$. By Theorem~\ref{th_Teta1} one has
\[
\MM(\tilde T) \leq \int_\Theta \MM(\ld\theta\rd)\,
d\tilde\eta(\theta), \qquad \MM(T') \leq \int_\Theta
\MM(\ld\theta\rd)\, d\eta'(\theta),
\]
and therefore, minding that $T=\tilde T+ T'$, we get
\begin{align*}
\MM(T) & \leq \MM(\tilde T)+ \MM(T')\leq \int_\Theta
\MM(\ld\theta\rd)\, d\tilde\eta(\theta) +\int_\Theta
\MM(\ld\theta\rd)\, d\eta'(\theta)\\
&=\int_\Theta \MM(\ld\theta\rd)\,  d \eta(\theta)=\MM(T),
\end{align*}
which implies that all the above inequalities are actually
equalities and hence shows the first equality of the thesis.
The second one is exactly the same calculation with
$\ell(\theta)$ instead of $\MM(\ld\theta\rd)$.
\end{proof}

\begin{lemma}\label{lm_Trestr}
Let $T\in \M_1(E)$ be a normal current and $\eta$ be such a transport
that $T=T_\eta$ and
%\begin{equation}\label{eq_Tom6a}
$\MM(T)
=  \int_\Theta \MM(\ld\theta\rd)\,  d \eta(\theta)$.
%\end{equation}
Then
$\mu_T=\mu_{\ld\theta\rd}\otimes \eta$, i.e.\
\begin{equation}\label{eq_Tom6}
\mu_T(e)
=  \int_\Theta \mu_{\ld\theta\rd} (e)\,  d \eta(\theta),
\end{equation}
and, moreover,
\begin{equation}\label{eq_Tom7}
  T\res \phi (f\,d\pi) = \int_\Theta
  \ld\theta\rd\res \phi (f\,d\pi)\,d\eta(\theta),
\end{equation}
for every Borel function $\phi\colon E\to \R$
and every $f\,d\pi\in D^1(E)$.
\end{lemma}

\begin{proof}
By Theorem~\ref{th_Teta1} one has
\[
%\begin{equation}\label{eq_Tom5}
\mu_T(e)
\leq  \int_\Theta \mu_{\ld\theta\rd}(e)\,  d \eta(\theta)
%\end{equation}
\]
for every Borel set $e\subset E$, while according to the assumptions
the latter estimate becomes an equality
for $e:=E$. Thus~\eqref{eq_Tom6} follows.
The relationship~\eqref{eq_Tom7} is just an
easy calculation
\begin{align*}
  T\res \phi (f\,d\pi)  = T(f\phi \,d\pi) &= \int_\Theta
  \ld\theta\rd (f\phi \,d\pi)\,d\eta(\theta)\\
  &=
  \int_\Theta
  \ld\theta\rd\res \phi (f\,d\pi)\,d\eta(\theta),
\end{align*}
which therefore concludes the proof.
\end{proof}

\section{Currents decomposable in curves}\label{sec_currdecomp1}

Our aim is now to prove the following theorem, which is the main result of
the paper.

\begin{theorem}\label{th_decompNorm1acycl}
Let $E$ be a complete metric space.
Then every acyclic normal one-dimensional real current $T$ %with bounded support
in $E$ is decomposable in curves, so that
in particular, there is a transport $\eta$
satisfying $T=T_\eta$, such that
relationships~\eqref{eq:claim1a} and~\eqref{eq:claim2} hold, and
$\eta$-a.e.\ $\theta\in \Theta$ is an arc.
\end{theorem}

\begin{remark}\label{rem_gentight4}
Let us emphasize that the statement of the above theorem is true in \emph{every} complete metric space,
since we assumed in the Introduction that the density character of every metric space is an Ulam number.
Without such an assumption this result still holds in view of Remark~\ref{rem_gentight2} when
$E$ is an arbitrary complete metric space and $\mu_T$ and $\mu_{\partial T}$ are tight measures, hence in particular
for every $T\in \M_1(E)$ once
$E$ is a Polish (i.e.\ complete separable) metric space.
\end{remark}

To this aim we first provide several technical statements.

First we prove a similar decomposition statement for
one-dimensional real {\em polyhedral\/} currents in a finite-dimensional
normed space.

%%%%%%%%%%%%%%%%%%%%%%%%%%%%%%%%%%%%%%%% versione 16.07.12 %%%%%%
\begin{lemma}\label{lm_measconv1}
Let $E$ be a finite-dimensional normed
space and $T\in \M_1(E)$ be an acyclic polyhedral current
over $E$, i.e.\
$T = \sum_{\nu=1}^N \theta_\nu T_\nu$, where
$\theta_\nu>0$, and
$T_\nu=\ld a_\nu,b_\nu\rd$ are currents associated to oriented segments
which may overlap only at endpoints.
Then there exists a Borel measure $\eta$ over $\Theta$ such that
$T=T_\eta$ and
relationships~\eqref{eq:claim1a} and~\eqref{eq:claim2} hold, while
$\eta$-a.e.\ $\theta\in \Theta$ is an arc.
\end{lemma}
%%%%%%%%%%%%%%%%%%%%%%%%%%%%%%%%%%%%%%%%

\begin{proof}
Let us call \emph{edges} the oriented segments $T_\nu=\ld a_\nu,b_\nu\rd$, $\nu=1,\dots,N$.
We say that an ordered finite collection of edges
$(T_{\nu_1},\ldots,T_{\nu_M})$, where
$T_{\nu_i}:= \ld a_{\nu_i}, b_{\nu_i}\rd$, $i=1,\ldots, M$,
is a \emph{path} in $T$,
if $b_{\nu_i}=a_{\nu_{i+1}}$ for $i=1,\ldots,M-1$. We say that such a path
is \emph{closed}, if also $b_{\nu_M}=a_{\nu_1}$.
Clearly an acyclic $T$ contains no closed paths.
Given a path in $T$, we can extend it
\emph{forward}, if there exists an edge $T_\nu$ of $T$ such that
$a_\nu=b_{\nu_N}$, and \emph{backward}, if there exists
and edge $T_\nu$ such that $b_\nu=a_{\nu_1}$.

Let $\bar \nu$ be such that $\theta_{\bar \nu} = \min\{\theta_1,\dots,\theta_N\}$
and consider the path $(T_{\bar \nu})$ with a single edge.
Then extend this path as much as
possible forward and backward. At each extension step the path cannot become
closed, hence the path is composed by all different edges. Since
there is only a finite number of edges in $T$, this extension process must
finish in a finite number of steps. We obtain in this way a \emph{maximal
  path} containing $T_{\bar \nu}$.
Let $(T_{\nu_1},\ldots,T_{\nu_M})$ be
this maximal path and consider the corresponding current
\[
  P_0 := \theta_{\bar \nu} \sum_{i=1}^M  T_{\nu_i}.
\]
Clearly, $P_0\leq T$ and
$\partial P_0 = \ld b_{\nu_M}\rd - \ld a_{\nu_1}\rd $.
Since the path is maximal, in $T$ there is no edge $T_\nu=\ld a_\nu, b_\nu\rd$
with endpoint $b_\nu=a_{\nu_1}$, and thus $(\partial P_0)^- = \theta_{\bar \nu}\ld a_{\nu_1}\rd$ is a subcurrent
of $(\partial T)^-$.
Analogously $(\partial P_0)^+ = \theta_{\bar \nu} \ld b_{\nu_M}\rd\leq (\partial T)^+$.

To represent $P_0$ as a measure on $\Theta$ we just consider the curve $\sigma_0$
representing the polygonal path $[a_{\nu_1},b_{\nu_1}]\circ\ldots \circ [a_{\nu_M},b_{\nu_M}]$ and the
Dirac measure $\eta_0 := \theta_{\bar \nu}\delta_{\sigma_0}$ to obtain $P_0 = T_{\eta_0}$.
Clearly $\eta_0(1)=(\partial P_0)^+$ and $\eta_0(0)=(\partial P_0)^-$.

The current $T':=T-P_0$ is itself a polyhedral acyclic current with $\partial T' \le \partial T$
(since $\partial P_0 \le \partial T$ as noted above).
Moreover $T'$ can be represented with strictly less edges than $T$ because the edge
$T_{\bar \nu}$ has been removed from $T$. Hence repeating the previous
construction with $T'$ in place of $T$ we find a subcurrent $P_1$
representing a path in $T'$ and such that $P_1 = T_{\eta_1}$
with $\eta_1(1)=(\partial P_1)^+ \le (\partial T)^+$ and $\eta_1(0)=(\partial P_1)^- \le (\partial T)^-$.
A finite number of such steps will eventually exhaust $T$ and yield a
decomposition $T=\sum_{i=0}^k P_i$ such that the corresponding
measure $\eta:= \sum_{i=0}^k \eta_i$ has the required properties.
\end{proof}

\begin{lemma}\label{lm_decompFIN0}
Let $E$ be a finite-dimensional normed space.
Then every acyclic normal  current $T\in \M_1(E)$
with bounded support
in $E$ is decomposable in curves, so that
in particular, there is a transport $\eta$
satisfying $T=T_\eta$, while
relationships~\eqref{eq:claim1a} and~\eqref{eq:claim2} hold, and
$\eta$-a.e.\ $\theta\in \Theta$ is an arc.
\end{lemma}

\begin{proof}
Combine Lemmata~\ref{lm_measconv1} and~\ref{lm_Tapprox0} with Proposition~\ref{prop_Teta2}.
\end{proof}

\begin{lemma}\label{lm_decompMAP1}
Let $E$ be a Banach space with metric approximation property.
Then every acyclic normal  current $T\in \M_1(E)$ %with bounded support
in $E$ is decomposable in curves, so that
in particular, there is a transport $\eta$
satisfying $T=T_\eta$, while
relationships~\eqref{eq:claim1a} and~\eqref{eq:claim2} hold, and
$\eta$-a.e.\ $\theta\in \Theta$ is an arc.
\end{lemma}

\begin{remark}\label{rem_gentight2}
The Lemma~\ref{lm_decompMAP1} is proven
under the set-theoretic assumption made in the Introduction. Without this assumption one has to
assume
that the measures $\mu_T$ and $\mu_{\partial T}$ are tight.
Then the statement of the Lemma is still true with the following argument added to the proof.
In fact, in the notation of the proof, one has $P_{n\#}\mu_{T}\weakto \mu_T$ in the narrow sense of measures
when $n\to \infty$, while the measures $P_{n\#}\mu_T$ are tight (in fact, they are concentrated over
the $\sigma$-compact set $P_n(\cup_\nu K_\nu)= \cup_\nu P_n(K_\nu)$),
so that in particular, the measures $P_{n\#}\mu_{T}$ are uniformly tight
by theorem~8.6.4 from~\cite{Bogachev06}. But, minding $\| | P_n | \|\leq 1$, we have
$\mu_{T_n}\leq P_{n\#}\mu_{T}$, which means that the measures $\mu_{T_n}$ are also uniformly tight.
Analogously, we have that the measures $\mu_{\partial T_n}=(\partial T_n)^+ + (\partial T_n)^-$ are also uniformly tight,
and hence so are the measures $(\partial T_n)^\pm$. The Proposition~\ref{prop_Teta2} in the proof may then be
invoked minding Remark~\ref{rem_gentight1}.
\end{remark}

\begin{proof}
Let
$\{K_\nu\}$ be an increasing sequence of compact subsets of $E$ such that
$\mu_T$ and $\mu_{\partial T}$ are concentrated on $\cup_\nu K_\nu$,
and let $P_\nu$ be a finite rank projection of norm one such that $\|P_\nu x-x\|\leq 1/\nu$ for all $x\in K_\nu$. Thus $P_\nu x\to x$ as $\nu\to \infty$ for all $x\in \cup_\nu K_\nu$. %$x\in \supp T\cup \supp \partial T$.

Consider first the case when $\supp T$ is bounded.
Let $T_n:= P_{n\#} T$. Clearly, $T_n\weakto T$ in the weak sense of currents.
In fact, for every $f\,d\pi\in D^1(E)$ we have
\begin{align*}
    |T(f\circ P_n\, d\pi\circ P_n) & -T(f\,d\pi)|   \leq |T(f\circ P_n\, d\pi\circ P_n)-T(f\circ P_n\,d\pi)|  +\\
    &\qquad |T(f\circ P_n\, d\pi)-T(f\,d\pi)| \\
    &\leq \int_E |f\circ P_n|\cdot |\pi\circ P_n- \pi|\, d\mu_{\partial T} + \Lip\, f\int_E |\pi\circ P_n- \pi|\, d\mu_T +\\
    &\qquad
    |T(f\circ P_n\, d\pi)-T(f\,d\pi)| \qquad \mbox{ by proposition~5.1 of~\cite{AmbrKirch00}}\\
    &
    \leq (\|f\|_\infty \Lip\, \pi +\Lip\, f \Lip\, \pi) \int_E  \|P_n x-  x\|\, d(\mu_{\partial T} +\mu_T) +\\
    &\qquad
    |T(f\circ P_n\, d\pi)-T(f\,d\pi)|,
\end{align*}
all the terms in the right-hand side tending to zero as $n\to \infty$ by the choice of $P_n$ (the first one by Lebesgue
theorem, recalling that $\|P_n x-  x\|\leq 2\|x\|$ and the support of $T$, and hence of $\partial T$, is bounded, while the last term
because $f(P_n(x))\to f(x)$ for $\mu_T$-a.e. $x\in E$).
%%%%% corretto, ma non serve piu`
% and $P_{n\#}\mu_{T}\weakto \mu_T$ in the narrow sense of measures
%when $n\to \infty$, so that in particular, the measures $P_{n\#}\mu_{T}$ are uniformly tight. But, minding $\| | P_n | \|\leq 1$, we have
%$\mu_{T_n}\leq P_{n\#}\mu_{T}$, which means that the measures $\mu_{T_n}$ are also uniformly tight.
Further,
%for any convergent subsequence of $\MM(T_n)$
we have
\begin{align*}
\MM(T)\leq & \liminf_n \MM(T_n) \leq \limsup_n \MM(T_n)\leq \MM(T),
%\MM(T)\leq & \lim_n \MM(T_n) = \lim_n\mu_{T_n}(E)\\
%& \leq \lim_n P_{n\#}\mu_{T}(E)= \mu_T(E)=\MM(T),
\end{align*}
since $\MM(T_n)\leq \MM(T)$,
and therefore $\MM(T_n)\to \MM(T)$ as $n\to \infty$.
Finally,
\[
(\partial T_n)^\pm=P_{n\#}(\partial T)^\pm-
P_{n\#}(\partial T)^+ \wedge P_{n\#}(\partial T)^-,
\]
and thus minding that $P_{n\#}(\partial T)^\pm\weakto (\partial T)^\pm$, we get $(\partial T_n)^\pm\weakto (\partial T)^\pm$
as $n\to \infty$ in the narrow sense of measures. It suffices then to
recall that $T_n$ are decomposable in curves (as currents over a finite-dimensional space by Lemma~\ref{lm_decompFIN0}) and apply Proposition~\ref{prop_Teta2}.

For the general case of a current $T$ with possibly unbounded support, we approximate
$T$ by a sequence $\{T_\nu\}\subset \M_1(E)$, such that each $T_\nu$ has bounded support and
$\MM(T_\nu-T)+\MM(\partial T_\nu -\partial T)\to 0$ as $\nu\to \infty$
(for this purpose just take $T_\nu := T\res g_\nu$ for a $g_\nu\in \Lip_1(E)$ with bounded support having
$0\leq g_\nu \leq 1$ and
$g_\nu=1$ on $B_\nu (0)$). Now $T_\nu$ is decomposable in curves as just proven, while the whole sequence
$\{T_\nu\}$ satisfies all the conditions of Proposition~\ref{prop_Teta2} (the only thing to verify is
$(\partial T_\nu)^\pm\weakto (\partial T)^\pm$ in the narrow sense of measures, which is true in view
of the corollary~8.4.8 from~\cite{Bogachev06}), and invoking the latter we conclude the proof.
\end{proof}

The following lemma is probably a folkloric fact which is however not easily found in the literature.

\begin{lemma}\label{lm_ellinfty_MAP}
$\ell^\infty$ has the metric approximation property.
\end{lemma}

\begin{proof}
One has to show the existence for every $\varepsilon>0$ and every finite
set $X\subset \ell^\infty$
of a finite-rank projection $T$ with $| \| T \| | \leq 1$ such that $\|T x-x\| <\varepsilon$
for all $x\in X$. In fact, then for every compact $K\subset \ell^\infty$
choosing a finite $\varepsilon$-net $X\subset K$, we get
for all $x\in K$, choosing $y\in X$ so that $\|x-y\| \leq \varepsilon$, the estimate
\begin{align*}
\| Tx-x\| & \leq \| Tx-T y \| + \|T y-y\| + \|y-x\| \leq 2\|x-y\| + \|T y-y\| \leq 3\varepsilon.
\end{align*}
We now construct a net of finite rank projections of norm one as follows.
Let $\Lambda$ be the directed set of all finite partitions of $\N$ ordered by refinement. For every partition $P\in \Lambda$,
$P= \{N_i\}_{i=1}^k$, $N_i\subset \N$ and all $N_i$ pairwise disjoint,
we define the finite rank projection $T=T_P$ by setting
$(Tx)_j:=x_{i_1}$ for all $j\in N_i$, where $i_1$ stands for the first
(i.e.\ lowest) index in $N_i$.
Clearly, for every $x\in \ell^\infty$ and every $\varepsilon >0$ there is a partition $P_{x,\varepsilon}\in \Lambda$ such that
$\|T_P x -x\| < \varepsilon$
for all $P\in \Lambda$ with $P>P_{x,\varepsilon}$
(such a partition is done by dividing the interval $[\inf x, \sup x]$ in subintervals $I_i$ of length not exceeding $\varepsilon$, and taking
$x^{-1} (I_i)$ to be the elements of $P$). Thus for a finite $X\subset \ell^\infty$
there is a partition $P_{X,\varepsilon}\in \Lambda$ such that
$\|T_P x -x\| < \varepsilon$ for all $x\in X$ and for all $P\in \Lambda$ with $P>P_{X,\varepsilon}$ (just take $P_{X,\varepsilon}> P_{x,\varepsilon}$ for all $x\in X$).
\end{proof}

Now we are able to prove Theorem~\ref{th_decompNorm1acycl}.

\begin{proof}[Proof of Theorem~\ref{th_decompNorm1acycl}]
Note that under the set-theoretic assumption made in the Introduction
$\mu_T$ is concentrated over $\supp\mu_T=:\supp T$, and the
value of $T(f\, d\pi)$ for $f\, d\pi\in  D^1(E)$ is completely determined by the restriction
of $f$ and $\pi$ to $\supp T$.
In fact, if $f\, d\pi\in  D^1(E)$, then
$T(f\, d\pi)= T(f\cdot 1_{\supp T}\, d\pi)$, and if
$\pi=0$ over $\supp T$, then
\[
T(f\, d\pi)= T(f\cdot 1_{\supp T}\, d\pi)=0,
\]
so that if $f^i\, d\pi^i\in  D^1(E)$, $i=1,2$,
with $f^1\, d\pi^1\res \supp T = f^2\, d\pi^2\res \supp T$, then
$T(f^1\, d\pi^1)=T(f^2\, d\pi^2)$.

Recalling that under the same set-theoretic assumption made in the Introduction the
set $\supp T\subset E$ is separable, we may just assume $E:=\supp T$ thus reducing to the case of
a complete separable metric space $E$.
Denote by $\jmath\colon E\to \ell^\infty$ an isometric embedding of $E$
into $\ell^\infty$.
Combining Lemma~\ref{lm_decompMAP1} with Lemma~\ref{lm_ellinfty_MAP} we get
that $\jmath_{\#} T$ is decomposable in curves, i.e.\ for some
transport $\eta'$ over $\Theta(\ell^\infty)$ one has
\begin{align*}
\jmath_{\#}T(f'\,d\pi') &=\int_{\Theta(\ell^\infty)} \ld\theta'\rd (f'\,d\pi')\,
d\eta'(\theta'),\\
\MM(\jmath_{\#}T) & = \int_{\Theta(\ell^\infty)} \MM(\ld\theta'\rd)\,  d \eta'(\theta'),
\mbox{ and }\\
  \eta'(1) &= (\partial \jmath_{\#}T)^+,\quad \eta'(0) = (\partial \jmath_{\#}T)^-,
\end{align*}
for all $f\,d\pi \in D^1(\ell^\infty)$, while
$\eta'$-a.e. $\theta'\in \Theta(\ell^\infty)$ is an arc.

Note that $\jmath$ induces the isometric imbedding $\imath\colon \Theta(E)\to \Theta(\ell^\infty)$ by the formula
\[
\imath(\theta)(t):= \jmath(\theta(t))
\]
for all $\theta\in \Theta(E)$ and $t\in [0,1]$.
Let $\Sigma\subset \jmath(E)\subset \ell^\infty$ be a set such that
$\mu_{\jmath_{\#} T}(\Sigma^c)=\jmath_{\#}\mu_T(\Sigma^c)=0$.
Then
by Lemma~\ref{lm_Trestr} for $\eta'$-a.e.\ $\theta'\in \Theta(\ell^\infty)$
one has that
$\mu_{\ld\theta'\rd}$ is concentrated over
$\Sigma$, hence
$\theta'(s) \in \Sigma$ for a.e.\ $s\in [0,1]$.
%By Lemma~\ref{lm_Trestr} for $\eta'$-a.e.\ $\theta'\in \Theta(\ell^\infty)$
%one has that
%$\mu_{\ld\theta'\rd}$ is concentrated over $\supp \jmath_{\#} T=\jmath(\supp T)$, hence
%$\theta'(s) \in \jmath(\supp T)$ for a.e.\ $s\in [0,1]$.
Let $\theta(s):=\jmath^{-1}(\theta'(s))$
for such $s$, and extend $\theta$ to the whole $[0,1]$ by continuity, so that
$\theta\in \Theta(E)$, and in particular, $\theta' =\imath(\theta)$.
Thus one has that $\eta'$ is concentrated over $\imath(\Theta(E))$, and hence we may define $\eta:=\imath^{-1}_{\#}\eta'$.
Note also that since $\eta'$-a.e.\ $\theta'\in \Theta(\ell^\infty)$ is an arc, then so is
$\eta$-a.e. $\theta =\imath^{-1}(\theta')\in \Theta(E)$.

For $f\,d\pi \in D^1(E)$ we define $f'\,d\pi' \in D^1(\ell^\infty)$ by setting $f'(x):= f(\jmath^{-1}(x))$, $\pi'(x):= \pi(\jmath^{-1}(x))$
for $x\in \jmath(E)$ and extending these functions to the whole $\ell^\infty$. Then
\[
\jmath_{\#}T(f'\,d\pi') = T(f\, d\pi)\mbox{ and }
\ld\theta'\rd (f'\,d\pi') = \jmath_{\#} \ld\theta\rd (f'\,d\pi') =
\ld\theta\rd (f\,d\pi).
\]
Hence, minding that $\eta'=\imath_{\#}\eta$, we get
\[
T(f\,d\pi) =\int_{\Theta(E)} \ld\theta\rd (f\,d\pi)\,
d\eta(\theta).
\]
Further, since
$\MM(T)=\MM(\jmath_{\#}T)$ and $\MM (\ld\theta'\rd)=\MM(\ld\theta \rd)$, one has
\[
\MM(T)  = \int_{\Theta(E)} \MM(\ld\theta'\rd)\,  d \eta(\theta).
\]
At last,
\[
  \eta(1) = (\imath^{-1}_{\#}\eta')(1) = \jmath^{-1}_{\#}(\eta'(1))=
  \jmath^{-1}_{\#}(\partial \jmath_{\#}T)^+= (\partial T)^+,
\]
and analogously,
\[
 \eta(0) = (\partial \jmath_{\#}T)^-,
\]
which concludes the proof.
\end{proof}

\appendix

\section{An application to optimal mass transportation}\label{sec_omt1}

In this section we provide an easy application of the representation result for acyclic metric currents
to optimal mass transportation problems in metric space. It is not our goal to present such applications in full generality,
but rather to illustrate the utility of the results proven in this paper.

Given two finite positive Borel measures $\varphi^+$ and $\varphi^-$
of equal total mass
with bounded (but not necessarily compact) support
in a given metric space $(E,d)$, the
classical Monge-Kantorovich optimal mass transportation problem in a metric space $(E, d)$ is that of finding
\begin{equation}\label{eq_acyclMK1}
    \inf\{\int_{E\times E} d(x,y)\, d\gamma(x,y)\,:\, \gamma \mbox{ admissible transport plan for $\varphi^+$ and $\varphi^-$} \},
\end{equation}
where by saying that $\gamma$ is admissible, we mean that $\gamma$ is a finite positive Borel measure over
 $E\times E$ satisfying the conditions on
marginals
\[
\pi^{\pm}_{\#} \gamma=\varphi^\pm,
\]
where $\pi^{\pm}\colon (x^+, x^-)\in E\times E \mapsto  x^\pm \in E$.
Recall that we are always assuming in this paper that finite positive Borel measures are tight
(otherwise we just impose the tightness condition on $\varphi^+$ and $\varphi^-$).
The above infimum is clearly attained under such conditions.
In fact in a minimizing sequence
$\{\gamma_\nu\}$ of admissible transport plans, all plans have the same total masses (equal to the total
mass of $\varphi^+$ and $\varphi^-$) and the sequence is uniformly tight, because
\[
\gamma_\nu ((K\times K)^c)\leq \gamma(K^c\times E)+ \gamma(E\times K^c)\leq 2\varepsilon
\]
whenever $K\subset E$ is a compact set such that $\varphi^\pm(K^c)\leq \varepsilon$. Hence, by
Prokhorov theorem
for nonnegative measures
(theorem~8.6.4 from~\cite{Bogachev06}) $\gamma_\nu$ admits a narrow convergent subsequence,
and therefore the existence of a minimizer follows from lower semicontinuity with respect to such a convergence
of integrals with nonnegative lower semicontinuous integrands (in our case the integrand is even continuous).
The value of the above infimum is usually denoted by
$W_1(\varphi^+,\varphi^-)$ and is called Wasserstein distance between
$\varphi^+$ and $\varphi^-$ (or Kantorovich-Rubinstein distance, which should be surely more correct
for historical reasons). Of course, to guarantee that $W_1(\varphi^+,\varphi^-)<+\infty$, extra conditions are required
(usually one imposes conditions on the moments of $\varphi^\pm$).

The following result then holds true.

\begin{theorem}\label{th_acyclMK1}
Assume that $E$ is a geodesic metric space (i.e.\ such that for every $(x,y)\in E\times E$ there is
a curve $\theta\in \Theta$ connecting $x$ to $y$ such that $d(x,y)=\ell(\theta)$), and, moreover,
there is a Borel map $q\colon \supp\varphi^+\times\supp\varphi^-\to \Theta(E)$ such that
$d(x,y)=\ell(q(x,y))$.
Then
\begin{equation}\label{eq_acyclMK2}
    W_1(\varphi^+,\varphi^-)=
    \min\{\MM(T)\,:\, T\in \M_1(E), \partial T=\varphi^+-\varphi^- \}.
\end{equation}
Moreover, if $T$ is a minimizer of the problem~\eqref{eq_acyclMK2}, then $T$ is acyclic, and if
$\eta$ is a transport such that $T=T_\eta$ for which conditions of Theorem~\ref{th_decompNorm1acycl} hold, then
$\gamma:=(e_0\times e_1)_\#\eta$
is a minimizer of~\eqref{eq_acyclMK1}, where $e_i(\theta):=\theta(i)$, $i=0,1$ for all $\theta\in \Theta(E)$.

Viceversa,
when $\gamma$ is a minimizer of~\eqref{eq_acyclMK1}, then setting $\eta:=q_{\#}\gamma$
(so that in particular $\eta$ is concentrated on a set of geodesics), we get
that $T=T_\eta$ satisfies conditions of Theorem~\ref{th_decompNorm1acycl} and minimizes~\eqref{eq_acyclMK2}.
\end{theorem}

\begin{remark}\label{rem_acyclMK_q}
The conditions of the above theorem are satisfied, for instance, in the following cases.
\begin{itemize}
\item[(i)]
When
$E$ is a separable geodesic space.
In fact, a map $q$ indicated in the conditions
exists in view of the Kuratowski-Ryll-Nardzewski measurable selection theorem~5.2.1 from~\cite{Srivast98}
 because the set
 \[
 \{(\theta(0), \theta(1),\theta)\subset \supp\varphi^+\times \supp\varphi^-
        \times C([0,1];E)\,:\, d(\theta(0),\theta(1))=\ell(\theta)\}
\]
is closed (here the space $C([0,1];E)$ is assumed to be equipped with the usual uniform metric factorized
by reparameterization of curves; further, it is assumed that $\ell(\theta):=+\infty$ for $\theta \in C([0,1];E)$
not rectifiable).
\item[(ii)]
When
$E$ is a Banach space (not necessarily separable).
One may set then $q(x,y):= [x,y]$,
where the curve $[x,y]$ is defined by
\[
[x,y](t):=(1-t)x+t y, \qquad t\in [0,1].
\]
\end{itemize}
\end{remark}

\begin{remark}\label{rem_acyclMKroot}
The above result is clearly false in generic metric spaces. In particular, if one takes
$E:=[0,1]$ equipped with the distance $d(x,y):=\sqrt{|x-y|}$, and $\varphi^+:=\delta_0$,
$\varphi^-:=\delta_1$, then by Theorem~\ref{th_decompNorm1acycl} there is no current
$T\in \M_1(E)$ such that $\partial T=\varphi^+-\varphi^-$ (because $\Theta(E)$ reduces to only constant curves),
so
\[
\inf\{\MM(T)\,:\, T\in \M_1(E), \partial T=\varphi^+-\varphi^- \} =\inf \emptyset = +\infty,
\]
while $W_1(\varphi^+,\varphi^-)=1$ in this case.
One has of course the same phenomenon if $E$ is just the discrete space
$E:=\{0,1\}$ with $d(0,1)\neq 0$ and with the same choice of $\varphi^\pm$.
This shows that in fact
the minimization problem
\[
\inf\{\MM(T)\,:\, T\in \M_1(E), \partial T=\varphi^+-\varphi^- \}
\]
corresponds better to the idea of mass transportation than the classical Monge-Kantorovich setting.
\end{remark}

\begin{proof}
Assume first $S\in \M_1(E)$ be such that
$\partial S=\varphi^+-\varphi^-$ and decompose $S=T+C$ with $C\leq S$, $\partial C=0$, and $T\leq S$ acyclic
by Proposition~\ref{prop_acycl1}.
%\[
%\MM(T)\leq \inf\{\MM(S)\,:\, S\in \M_1(E), \partial S=\varphi^+-\varphi^- \} +\varepsilon
%\]
%for some $\varepsilon>0$.
%Up to decreasing even more the mass of $T$ by deleting cycles, we may assume that $T$ is acyclic.
If $\eta$ is a transport such that $T=T_\eta$ for which conditions of Theorem~\ref{th_decompNorm1acycl} hold, then
setting $\gamma:=(e_0\times e_1)_\#\eta$, we have
that $\gamma$ is admissible and
\begin{equation}\label{eq_acyclMK3}
\begin{aligned}
\MM(S)\geq \MM(T)& =\int_{\Theta(E)} \ell(\theta)\, d\eta(\theta)= \int_{\Theta(E)} d(\theta(0),\theta(1))\, d\eta(\theta)\\
& =
\int_{E\times E}d(x,y)\,d \gamma(x,y)\geq W_1(\varphi^+,\varphi^-).
\end{aligned}
\end{equation}

Further,
 let
$\gamma$ be a minimizer of~\eqref{eq_acyclMK1}, $\eta:=q_{\#}\gamma$. We get
then
for $T=T_\eta$
\begin{equation}\label{eq_acyclMK4}
\begin{aligned}
W_1(\varphi^+,\varphi^-) & =\int_{E\times E}d(x,y)\,d \gamma(x,y)= \int_{E\times E}\ell(q(x,y))\,d \gamma(x,y)\\
& = \int_{\Theta(E)} \ell(\theta)\, d\eta(\theta)\geq
\MM(T_\eta),
\end{aligned}
\end{equation}
the latter inequality being due to Theorem~\ref{th_Teta1}. Combined with~\eqref{eq_acyclMK3} this gives
the optimality of $T$ for~\eqref{eq_acyclMK2}. In particular, equality holds in~\eqref{eq_acyclMK4}, so that
$\eta$ satisfies conditions of Theorem~\ref{th_decompNorm1acycl}.

Finally, it remains to observe that every minimizer $T$ of~\eqref{eq_acyclMK2} is acyclic since
deleting cycles decreases the mass without changing the boundary of a current.
\end{proof}

Theorem~\ref{th_acyclMK1}
shows the equivalence of three different descriptions of optimal mass transportation: the classical one in terms
of transport plans $\gamma$
proposed by Kantorovich which gives only the information on ``who goes where'' (i.e.\ only staring points and endpoints of transport paths), the one in terms of transports $\eta$
%(roughly speaking the same concept is called ``dynamical coupling'' in~\cite{Villani06omt})
(which is the most precise one since it gives the full information on paths covered by infinitesimal masses during transportation), and the intermediate one in terms of currents $T$ giving the information on the total flow of mass.
Of course, the respective claims can be obtained also without using representation Theorem~\ref{th_decompNorm1acycl} for acyclic currents.
For instance the
inequality
\begin{equation}\label{eq_acyclMK5}
    W_1(\varphi^+,\varphi^-)\leq
  \MM(T)
\end{equation}
for every $T\in \M_1(E)$ satisfying $\partial T=\varphi^+-\varphi^-$ may be seen as a consequence of
Kantorovich duality
\[
  W_1(\varphi^+,\varphi^-)=
  \sup\left\{\int_E f\,d(\varphi^+-\varphi^-)\,:\, f\in \Lip_1(E)\right\}
\]
coupled with the obvious relationship
\[
\int_E f\,d(\varphi^+-\varphi^-) =\partial T(f)= T(df)\leq \MM(T)
\]
whenever $f\in \Lip_1(E)$.
Together with~\eqref{eq_acyclMK4}
which is proven without use of the representation Theorem~\ref{th_decompNorm1acycl}
(see the proof of Theorem~\ref{th_acyclMK1}) this shows
the equality~\eqref{eq_acyclMK2}.

We call a transport $\eta$ admissible, if $(e_0)_{\#}\eta=\varphi^+$, $(e_1)_{\#}\eta=\varphi^-$.
The construction used to prove~\eqref{eq_acyclMK4} shows also the
existence of an admissible transport $\eta'$ such that
\[
W_1(\varphi^+,\varphi^-)=\int_{\Theta(E)}\ell(\theta)\, d\eta'(\theta),
\]
while using~\eqref{eq_acyclMK5} for $T=T_\eta$ for an arbitrary admissible transport $\eta$ and
employing Theorem~\ref{th_Teta1}, we have
\[
W_1(\varphi^+,\varphi^-)\leq \int_{\Theta(E)}\ell(\theta)\, d\eta(\theta),
\]
so that in fact we have that $W_1(\varphi^+,\varphi^-)$ is also equal to the minimum among all admissible
transports $\eta$ of the functional $\eta\mapsto \int_{\Theta(E)}\ell(\theta)\,d\eta(\theta)$.
In this way one proves that the representation claimed in Theorem~\ref{th_decompNorm1acycl}
is true for optimal (i.e.\ mass minimizing) currents, and thus all this machinery avoiding the use of representation Theorem~\ref{th_decompNorm1acycl}
in a sense amounts to proving it ``manually'' only
for such currents, which are of course automatically acyclic.
Thus, once
proven for all acyclic currents, Theorem~\ref{th_decompNorm1acycl}
becomes an easy and natural alternative to such a machinery (observe that our proof of Theorem~\ref{th_acyclMK1}
is just few lines). Moreover, a similar result can be proven almost identically with the help of Theorem~\ref{th_decompNorm1acycl} for so-called branched optimal
transportation (see~\cite{BernCasMor08_book} for the introduction to the subject), which however goes beyond
the purely illustrative scope of this section.

\section{Metric currents}\label{sec_curr0}

Throughout the paper we are extensively using the notion of currents
with finite mass due to Ambrosio and
Kirchheim~\cite{AmbrKirch00}.

For a metric space $E$ we denote
\[
D^k(E):=\left\{ (f, \pi_i, \ldots, \pi_k)\,:\, f\in \Lip_b(E), \pi\in \Lip(E;\R^k)
\right\}.
\]
The $k$-tuples $\omega=(f, \pi_i, \ldots, \pi_k)\in D^k(E)$ will be
further called $k$-dimensional differential forms.
For the form $\omega=(f, \pi_i, \ldots, \pi_k)\in D^k(E)$ we will
adopt the notation
\[
\omega = f\,d\pi_1\wedge d\pi_2\wedge\ldots\wedge d\pi_k = f\,d\pi.
\]
The exterior derivative operator $d\colon D^k(E)\to D^{k+1}(E)$ is then
defined by
\[
d(f\,d\pi_1\wedge d\pi_2\wedge\ldots\wedge d\pi_k)  :=
1\, d f\wedge d\pi_1\wedge d \pi_2\wedge\ldots\wedge d\pi_k.
\]
Also, given an arbitrary Lipschitz map $\phi\colon F\to E$,
with $F$ metric space, one defines
the pull-back operator
$\phi^{\#}\colon D^k(E)\to D^k(F)$
by setting
\[
\phi^{\#}( f\,d\pi)
:=
f\circ\phi\,d\pi\circ\phi.
\]

\begin{definition}\label{def_AKcurrent}
A %subadditive positively homogeneous
functional $T\colon D^k(E)\to \R$ is called real $k$-dimensional metric
current (called further for simplicity current) over $E$, if the following conditions
hold:
\begin{itemize}
  \item [(linearity)] $(f,\pi_1,\ldots,\pi_k)\mapsto T(f,\pi_1,\ldots,\pi_k)$
  is multilinear, i.e.\ linear in $f$ and in each of $\pi_i$, $i=1,\ldots, k$,
  \item [(continuity)] $T(f\,d\pi_\nu)\to T(f\,d\pi)$ whenever $\pi_\nu\to \pi$
  pointwise in $\Lip(E;\R^k)$, as $\nu\to \infty$, and have uniformly bounded Lipschitz constants,
  \item [(locality)]
  $T(f\, d\pi)=0$ whenever for some $i\in \{1,\ldots, k\}$ the function $\pi_i$ is constant in the neighborhood
  of $\{f\neq 0\}$,
  %if
%\[
%\{x\in E\,:\, f(x)\neq 0\}\subset \cup_{i=1}^k B_i,
%\]
%  where each $B_i\subset E$ is a Borel set such that $\pi_i$ is constant over $B_i$,
%  $i=1,\ldots, k$, then $T(f\, d\pi)=0$,
  \item [(finite mass)] one has for some finite positive Borel measure $\mu$ over $E$ the estimate
\begin{equation}\label{eq_curr_finmass0}
|T(f\,d\pi)|\leq \prod_{i=1}^k \Lip(\pi_i) \int_E |f|\, d\mu
\end{equation}
valid
for every $f\in \Lip_b(E)$, $\pi\in \Lip(E,\R^k)$, with the convention
\[
\prod_{i=1}^k \Lip(\pi_i) :=1,
\]
if $k=0$.
\end{itemize}
\end{definition}

The mass measure $\mu_T$ is defined to be the minimum over all
finite Borel measures $\mu$ satisfying~\eqref{eq_curr_finmass0}, and the
total mass of a current $T$ is defined by
$\MM(T):=\mu_T(E)$.
The support $\supp T$ of a real $k$-dimensional metric current $T$ with finite mass
is defined as the support of $\mu_T$.
%In what follows we will be always
%dealing with real $k$-dimensional metric currents
%with
%%compact supports and
%finite mass in $E$.
The set of such currents will be denoted by
$\M_k(E)$.
The mass functional $\MM$ is easily seen to define a norm in $\M_k(E)$.

We will say that a sequence of
currents $\{T_\nu\}\subset \M_k(E)$ converges
weakly to a current $T\in \M_k(E)$, and write
$T_\nu\rightharpoonup T$,  if
$T_\nu(\omega) \to T(\omega)$
as $\nu\to \infty$,
for every $\omega\in D^k(E)$. The mass is known to be
lower semicontinuous with respect to weak convergence of
currents~\cite{AmbrKirch00}.

Clearly, every metric current
$T\in \M_k(E)$ may be extended by continuity
from the space of forms $D^k(E)$ to the larger space
of $(k+1)$-tuples
$(f, \pi_i, \ldots, \pi_k)$, where
$\pi\in \Lip(E;\R^k)$, while
$f\colon E\to \R$ is a bounded Borel function on $E$.
Thus, whenever $E$ is a complete metric
space, then every
$T\in \M_0(E)$
is represented by some signed Borel measure of finite total variation over $E$
(given by the set function $B\mapsto T(1_B)$ for every Borel set
$B\subset E$, where $1_B$ stands for the characteristic function of $B$).
Hence, when necessary, we will
always identify a $T\in \M_0(E)$ with the respective signed measure.
Note that the mass $\MM$ over $\M_0(E)$ is nothing else
than the total variation norm $\|\cdot\|$ over the space of such measures
on $E$.

If $T\in \M_k(E)$ and $\omega= g\, d\tau\in D^m(E)$, $m\leq k$,
we define the
restricted metric current $T\res\omega \in \M_{k-m}(E)$ by the formula
\[
T\res\omega (f\, d\pi):= T(fg, \tau_1, \ldots, \tau_m,
\pi_1, \ldots, \pi_{k-m}) \mbox{ for all } f\, d\pi\in D^{k-m}(E).
\]
Since $T$ is assumed to have finite mass, then in the above formula
one may admit in place of $f$ and $g$ any bounded Borel functions.
In particular, whenever $\omega=1_B\in D^0(E)$ for some Borel
set $B\subset E$, we will simply write $T\res B$ for $T\res\omega$.

The boundary $\partial T$ of a $k$-dimensional current $T$ is
a $(k-1)$-dimensional current
defined by the formula
\[
\partial T(\omega):= T(d\omega) \mbox{ for all } \omega\in D^{k-1}(E).
\]
Further, for an arbitrary Lipschitz map $\phi\colon F\to E$,
with $F$ metric space, we define
the push-forward operator $\phi_{\#}\colon \M_k(F)\to \M_k(E)$ on currents
by setting
\[
(\phi_{\#}T)(\omega):= T(\phi^{\#}\omega) \mbox{ for all } \omega\in D^k(E).
\]
%Clearly, one has
%\[
%\mu_{\phi_{\#}T}\leq (\Lip\,\phi)^k \phi_{\#}\mu_T,
%\]
%hence in particular
%\begin{equation}\label{eq_lip_mass1}
% \MM(\phi_{\#}T)\leq (\Lip\,\phi)^k \MM(T),
%\end{equation}
%with the equalities when $k=0$.
%Further, since the exterior derivative and the pull-back operator on
%differential forms commute by constructions, then so do
%the boundary and the push-forward operator on currents, i.e.\
%\begin{equation}\label{eq_dpull_comm}
%    \partial (\phi_{\#}T)=\phi_{\#}(\partial T).
%\end{equation}

We say that $T$ is a
{\em normal current\/}, if $\MM(T)<+\infty$ and $\MM(\partial
T)<+\infty$.
It is worth remarking that in a finite-dimensional Euclidean space
$E=\R^n$ every normal current (in the sense of metric currents)
by theorem~11.1 from~\cite{AmbrKirch00}
may be identified via a natural isomorphism
with a Whitney normal current.

If $E$ is a normed
space, we call oriented segment $\ld a,b\rd$ the curve $\theta$ (or, to be more precise, the
equivalence class of curves in $\Theta(E)$) that may be parameterized by
$\theta(t):=(1-t)a+tb$, $t\in [0,1]$.
We identify oriented segments with one-dimensional currents associated with them.
We further call $T\in \M_1(E)$ \emph{polyhedral} current, if
$T = \sum_\nu \theta_\nu T_\nu$, where
$\theta_\nu>0$, and
$T_\nu$ are currents associated to oriented segments
$T_\nu=\ld a_\nu,b_\nu\rd$ which may overlap only at the endpoints.

The following easy statement regarding weak convergence of metric currents has to be mentioned.

\begin{lemma}\label{lm_AKtight1}
Let $T_\nu\in \M_1(E)$, $T_\nu\rightharpoonup T$ in the weak sense of currents
and $\MM(T_\nu)\to \MM(T)$ as $\nu\to \infty$.
Then $\mu_{T_\nu}\rightharpoonup \mu_T$ in the narrow sense of measures
and in particular, the sequence of measures $\{\mu_{T_\nu}\}$ is uniformly tight.
\end{lemma}

\begin{remark}\label{rem_gentight3} %%% gentight
The
conclusion on uniform tightness of $\{\mu_{T_\nu}\}$
is true by theorem~8.6.4 from~\cite{Bogachev06}
if the measures $\mu_{T_\nu}$ and $\mu_T$ are tight
(which is automatically satisfied once one makes the set-theoretical assumption mentioned in the Introduction).
\end{remark}

\begin{proof}
One has $\mu_{T_\nu}(E)\to \mu(E)$ and
\[
\mu_T(U)\le \liminf_\nu \mu_{T_\nu}(U)
\]
for every open $U\subset E$, and therefore $\mu_{T_\nu}\rightharpoonup \mu_T$ in the narrow sense of measures
by theorem~8.2.3 from~\cite{Bogachev06}. The uniform tightness of $\{\mu_{T_\nu}\}$ follows then from Prokhorov theorem
for nonnegative measures
(theorem~8.6.4 from~\cite{Bogachev06}).
\end{proof}

%\begin{lemma}\label{lm_AKtight2}
%Let $T_\nu\in \M_1(E)$, $T_\nu\rightharpoonup T$ in the weak sense of currents
%as $\nu\to \infty$ and $\MM(\partial T_\nu)\leq C$ for some $C>0$.
%Then the sequence of measures $\{\mu_{\partial T_\nu}\}$ is tight.
%\end{lemma}
%
%\begin{proof}
%One has
%$\partial T_\nu\rightharpoonup \partial T$ in the weak sense of currents, which means, if
%we look at $\partial T_\nu$ as measures, that
%$\langle \partial T_\nu, f\rangle= \partial T_\nu (f)$ is fundamental for all
%$f\in \Lip_b(E)$ and hence,
%minding $\MM(\partial T_\nu)\leq C$, also for all $f\in C_b(E)$.
%Hence, the claim follows from corollary~8.6.3 from~\cite{Bogachev06}.
%\end{proof}

\section{Polyhedral approximation in finite dimensions}%OK

This section contains an auxiliary assertion on approximation of currents over a finite-dimensional
normed space $E$.
In the case when $E=\R^n$ is Euclidean, analogous results
can be found, e.g., in~\cite{Feder}[4.1.23,4.2.24] (our result is a bit more precise for one-dimensional currents and tailored for our purposes so as to be used in combination with Lemma~\ref{lm_measconv1} and Proposition~\ref{prop_Teta2}).
Throughout this section $\dim E=n$, and $E$ is assumed to be endowed with
some norm $\|\cdot\|$, while the notation $\R^n$ will stand for the same space endowed with the Euclidean norm $|\cdot|$.
We denote by $\F(T)$ the \emph{flat norm} of $T$ defined by
\[
\F(T):=\inf\{\MM(A)+\MM(B)\,:\, A\in \M_k(E),\, B\in \M_{k+1}(E),\, A+\partial B = T\}.
\]

%%%%%%%%%%%%%%%%%%%%%%%%%%%%%%%%%%%%%%%%%%%%%%%%%%
\begin{lemma}\label{lm_Tapprox0}
Let $T\in \M_1(E)$ be a normal current with compact support
over the finite-dimensional space $E$, and $r>0$ be such that
$\supp T\subset B_r(0)\subset E$.
Then there is a sequence of one-dimensional real polyhedral currents $T_\nu$ over $E$
with $\supp T_\nu\subset B_{2r}(0)$,
which converge to $T$ in the flat norm (in particular, weakly), i.e.\ $\F(T_\nu-T)\to 0$, while
$(\partial T_\nu)^\pm \weakto (\partial T)^\pm$
in the $*$-weak sense of measures (in particular,
$\MM(\partial T_\nu)\to \MM(\partial T)$) and $\MM(T_\nu)\to \MM(T)$ as
$\nu\to \infty$.
If $T$ is acyclic, one may choose $T_\nu$ to be acyclic too.
\end{lemma}
%%%%%%%%%%%%%%%%%%%%%%%%%%%%%%%%%%%%%%%%%%%%%%%%%%

%\begin{remark} %%%corretto, ma ovvio
%The sequence $\{T_\nu\}$ constructed in the above Lemma~\ref{lm_Tapprox0}
%automatically satisfies all the conditions of Proposition~\ref{prop_Teta2}.
%\end{remark}

\begin{proof}

{\em Step 1}.
We first show adapting the proof of~\cite{Feder}[4.1.23] that $T$ may be approximated in flat norm by a sequence
of polyhedral currents $S_\nu\in \M_1(E)$
supported over $B_r(0)$
with $\MM(S_\nu)\to \MM(T)$
as
$\nu\to \infty$.
For this purpose, first, by choosing the approximate identity
\[
\varphi_\eps(x):=\frac{1}{\eps^n} \varphi(\frac{x}{\eps}),
\]
where
$\varphi\in C_0^\infty(\R^n)$, $\varphi\geq 0$, $\|\varphi\|_1=1$,
$\eps>0$, define $T_\eps\in \M_1(E)$ by setting
\[
T_\eps(\omega):=\int_{\R^n} (\tau_{x\#} T)(\omega)\varphi_\eps(-x)\, dx
\]
for all $\omega\in D^1(E)$. Once one considers $T$ and $T_\eps$ as currents over $\R^n$ (with Euclidean norm), by~\cite{Feder}[4.1.18] one gets
$\F_2(T_\eps-T)\leq \eps (\MM_2(T)+\MM_2(\partial T))$,
where $\F_2$ and $\MM_2$ stand for the flat norm and mass over Euclidean flat chains.
Hence $\F(T_\eps-T)\to 0$ as $\eps\to 0^+$. On the other hand,
\[
\MM(T_\eps)\leq \int_{\R^n} \MM(\tau_{x\#} T)\varphi_\eps(-x)\, dx=
\int_{\R^n} \MM(T)\varphi_\eps(-x)\, dx=\MM(T),
\]
which combined with lower semicontinuity of mass gives $\MM(T_\eps)\to \MM(T)$ as $\eps \to 0^+$.
Analogously, one has
\[
\partial T_\eps(\omega):=\int_{\R^n} (\tau_{x\#} \partial T)(\omega)\varphi_\eps(-x)\, dx
\]
for every $\omega\in D^{0}(E)$, hence
$\MM(\partial T_\eps)\to \MM(\partial T)$ as $\eps \to 0^+$.
Also, clearly, $\supp T_\eps %\cup\supp \partial T_\eps
\subset B_r(0)$ once $\eps>0$ is
sufficiently small.
But
\[
T_\eps(f\,d\pi)=\int_{\R^n} f(x) (\nabla\pi(x),l)\, dx
\]
for some integrable vector field $l=l_\eps\colon \R^n\to \R^n$ (cfr.~proposition~6.1 combined with theorem~1.3 in~\cite{Williams10}), and therefore this reduces the proof of the desired assertion to the case of $T$ having exactly such form.

We may thus assume now $T(f\,d\pi):=\int_{\R^n} f(x) (\nabla\pi(x),l)\, dx$
for some integrable vector field $l\colon \R^n\to \R^n$.
Since $\|\nabla\pi(x)\|'\leq \Lip\, \pi $ for all $x\in E$,
where $\|\cdot\|'$ stands for the norm in the space $E'$ dual to $E$, one clearly has $\mu_T\leq \|l\|\,dx$.
Moreover, in fact the equality $\mu_T= \|l\|\,dx$ holds. It is clearly enough to prove this for the case
$l$ is a simple (i.e.\ finite valued) function, that is, $l=\sum_{i=1}^m l_i 1_{E_i}$ for some constants $l_i\in E$ and
Borel sets $E_i\subset E$. In this case just take $l_i'\in E'$ be such that $(l_i', l_i)=\|l_i\|$, $\|l_i'\|'=1$ and $\pi_i\colon E\to \R$
be affine functions such that $\nabla \pi_i=l_i'$, hence $\Lip\,\pi_i\leq 1$. Then
\[
\mu_T(e) =\MM(T\res e)\geq \sum_{i=1}^m T(1_{E_i\cap e}\,d\pi_i)=\sum_{i=1}^m \mathcal{L}^n(E_i\cap e)\|l_i\|=
\int_e\|l\|\, dx,
\]
where $\mathcal{L}^n$ stands for the $n$-dimensional Lebesgue measure in $E$.

Approximating $l$ by piecewise constant vector fields $l_k$
which are constant over a finite number of rectangles $R^k_i\subset E$, with one side of the rectangle parallel to the direction of $l_k$ inside $R^k_i$, the approximation being intended in the sense $\int_{R^n}|l-l_k|\,dx\to 0$ as $k\to +\infty$ (so that the currents $T_k$ defined by
$T_k(f\,d\pi):=\int_{\R^n} f(x) (\nabla\pi(x),l_k)\, dx$, converge to $T$ in mass), we reduce the problem to the case
\[
T(f\,d\pi):=\int_{R} f(x) (\nabla\pi(x),l)\, dx,
\]
where $R\subset \R^n$ is a rectangle and $l(x)$ is constant and parallel to one of the sides of $R$. Let $[a,b]$ be a side of $R$ parallel to $l$
and directed in the same direction as $l$ (i.e.\
with the vector $b-a$ having the direction of $l$), and
 consider the $(n-1)$-dimensional face $S$ of $R$
perpendicular to $l$ such that $a\in S$. Dividing $S$ by a uniform dyadic grid
with nodes $\{x_i\}_{i=1}^{N_\nu}$, with $N_\nu=2^{\nu}-1$, and setting $\theta_i(t):=
%(1-t)x_i +t(x_i+l)=
x_i+tl$ for $t\in [0,1]$, we let
\[
S_\nu := \frac{\mathcal{L}^n(R)\cdot \|l\|}{N_\nu\|b-a\|} \sum_{i=1}^{N_\nu} \ld\theta_i\rd,
\]
so that, minding
$\MM(\ld\theta_i\rd)=\ell(\theta_i)= \|b-a\|$, we have
 $\MM(S_\nu)= \mathcal{L}^n(R)\cdot \|l\| =\MM(T)$.
%  and
% $\MM(\partial S_\nu)= 2\mathcal{L}^n(R)\cdot \|l\|/\|b-a\|$.
Clearly, one has $\F(S_\nu-T)\to 0$ as $\nu\to\infty$ (e.g.\ one may refer
to the fact that $\MM_2(T)=\mathcal{L}^n(R)\cdot |l|=\MM_2(S_\nu)$, and
$\F_2(S_\nu-T)\to 0$ as $\nu\to \infty$).

{\em Step 2}. Let $S_\nu$ be a sequence constructed in Step~1 of the proof.
We as usual identify finite purely atomic measures with
zero-dimensional polyhedral currents.
Recall that $(\partial T)^+$ has the same total mass as
$(\partial T)^-$ since
\[
\partial((\partial T)^+-(\partial T)^-)=\partial \partial T=0.
\]
Let $\phi_\nu^\pm$ be purely atomic measures
with compact support over $B_r(0)$,
having the same total mass as
$(\partial T)^\pm$
(so that in particular, $\MM(\phi_\nu^+-\phi_\nu^-)=\MM(\partial T)$)
and such that
\[
\F(\phi_\nu^\pm-(\partial T)^\pm)\to 0
\]
as $\nu\to \infty$
(recall that in fact, $\F_2$, and hence also $\F$, metrizes $*$-weak topology over the set of finite Borel measures concentrated over a compact subset of $E$).
We now act as in the proof of~\cite{Feder}[4.2.24].
Mind that
$\F(\partial T-\partial S_\nu)\leq \F(T-S_\nu)$ and hence
\[
\F(\phi_\nu^+-\phi_\nu^- -\partial S_\nu) \leq
\F(\phi_\nu^+-\phi_\nu^- -\partial T) +
\F(T-S_\nu) \to 0
\]
as $\nu\to \infty$.
Using~\cite{Feder}[4.2.23] choose now one-dimensional polyhedral currents
$Y_\nu$ with $\supp Y_\nu\subset B_r(0)$ such that
\[
\MM(\phi_\nu^+-\phi_\nu^- -\partial S_\nu-\partial Y_\nu) +\MM(Y_\nu)\to 0,
\]
and set
$T_\nu:=S_\nu+Y_\nu$, so that
$\supp T_\nu \subset B_{2r}(0)$.
One  has then
\begin{align*}
|\MM(T_\nu)-\MM(T)|&\leq |\MM(S_\nu)-\MM(T)| + \MM(Y_\nu)\to 0,\\
|\MM(\partial T_\nu)-\MM(\partial T)| &= |\MM(\partial S_\nu+\partial Y_\nu)-\MM(\partial T)|\\
&\leq \MM(\phi_\nu^+-\phi_\nu^- -\partial S_\nu-\partial Y_\nu)+ |\MM(\phi_\nu^+-\phi_\nu^-)-\MM(\partial T)|\\
&= \MM(\phi_\nu^+-\phi_\nu^- -\partial S_\nu-\partial Y_\nu) \to 0,
\end{align*}
while
\[
\F(T_\nu-T)\leq F(T_\nu-T)+\F(Y_\nu)\leq F(T_\nu-T)+\MM(Y_\nu)\to 0
\]
as $\nu\to \infty$.
Finally, viewing $\partial T_\nu$ and $\partial T$ as signed measures,
we have that the total variations of the former are uniformly bounded
and converge to that of the latter, and therefore
$(\partial T_\nu)^\pm \weakto (\partial T)^\pm$
in the $*$-weak sense of measures  as
$\nu\to \infty$
by corollary~8.4.8 of~\cite{Bogachev06}.

{\em Step 3}. If $T$ is acyclic, we modify $T_\nu$ in the following way.
Let $C_\nu$ be the cycle of $T_\nu$ given by Proposition~\ref{prop_acycl1}
such that $T_\nu':=T_\nu - C_\nu$ is acyclic.
Up to a subsequence (not relabeled), $C_\nu\weakto C$ as $\nu\to \infty$.
Hence, by Lemma~\ref{lm:compconv}, $\MM(C_\nu)\to \MM(C)$ as $\nu\to \infty$
and $C$ is a
cycle of $T$. Since the only cycle of $T$ is zero we conclude that
 $\MM(C_\nu)\to 0$, which means that $T_\nu'\weakto T$ and $\MM(T_\nu')\to
 \MM(T)$ as $\nu\to \infty$.

It remains to observe that
$T_\nu'\leq T_\nu$, and since
$T_\nu =\sum_{i=1}^{m_\nu} \alpha_{i,\nu}\ld\Delta_{i,\nu}\rd$, where
$\alpha_{i,\nu}\in \R$ and $\Delta_{i,\nu}\subset E$ are segments which may overlap only at the endpoints, then
\begin{equation}\label{eq_TNDel1}
T_\nu'\res \Delta_{i,\nu} \leq T_\nu\res \Delta_{i,\nu}
\end{equation}
 by Remark~\ref{rem_subcurr4}
for all $i=1,\ldots, m_\nu$. Further, one has
\begin{equation}\label{eq_TNDel2}
\partial (T_\nu'\res \Delta_{i,\nu}) \leq \partial (T_\nu\res \Delta_{i,\nu})
\end{equation}
for all $i=1,\ldots, m_\nu$,
since otherwise by Lemma~\ref{lm_polysub} one would have
that $\partial (T_\nu'\res \Delta_{i,\nu})$ charges the interior of a segment
$\Delta_{i,\nu}$  for some $i=1,\ldots, m_\nu$, which would contradict
$\partial T_\nu'=\partial T_\nu$  (the latter measure being concentrated only at the endpoints of the segments
 $\Delta_{i,\nu}$).
Therefore, from~\eqref{eq_TNDel1} and~\eqref{eq_TNDel2} with the help of Lemma~\ref{lm_polysub} we get
\[
T_\nu'\res \Delta_{i,\nu}=\beta_{i,\nu}\ld\Delta_{i,\nu}\rd
\]
for some $\beta_{i,\nu}\in \R$, and hence $T_\nu'$ is still polyhedral.
\end{proof}

\begin{lemma}\label{lm_polysub}
Let $E$ be an arbitrary metric space.
If $T=\sigma_{\#}\ld a,b\rd$, where $\sigma$ is injective,
$\{a,b\}\in \R$ and
$S\le T$, then $S=T\res \lambda$ for some Borel function $\lambda\colon E\to [0,1]$. Further, either
$\partial S\le \partial T$, which happens if and only if
$\lambda\in [0,1]$ is constant over $\supp T$, or $\partial S$ considered as a measure
charges $\sigma([a,b])\setminus\{\sigma(a), \sigma(b)\}$.
\end{lemma}

\begin{proof}
Assume without loss of generality that $\sigma$ is parameterized by arclength (in particular, then $a=0$), so that $\sigma$ is an isometry
between $[a, b]$ and $\sigma([a,b])$.
Denote
\begin{align*}
\tilde{S} :=\sigma^{-1}_{\#} S,\qquad\qquad
\tilde{T}:=\ld a, b \rd ,
\end{align*}
so that in particular
\[
\tilde{T}-\tilde{S}=\sigma^{-1}_{\#} (T-S).
\]
Since $S\leq T$, then
by Remark~\ref{rem_subcurr4} one has
$\mu_S\leq \mu_T$ and hence $\mu_S=\lambda\mu_T$ for some
Borel function $\lambda$ satisfying $0\leq \lambda\leq 1$.
Minding now that $\sigma$ is an isometry,
we get
\begin{align*}
\mu_{\tilde{S}} &= \sigma^{-1}_{\#} \mu_{S}=(\lambda\circ \sigma)\\
\mu_{\tilde{T}-\tilde{S}}  & =  \sigma^{-1}_{\#} \mu_{T-S}=(1-\lambda\circ \sigma)
\mu_{\tilde{T}},
\end{align*}
where $\mu_{\tilde{T}}={\mathcal L}^1\res [a, b]$.
This means $\mu_{\tilde{S}}+\mu_{\tilde{T}-\tilde{S}}=\mu_{\tilde{T}}$, or, in other words,
$\tilde{S}\leq \tilde{T}$. Now, since $\tilde{S}$ and $\tilde{T}$ are one-dimensional currents
in $\R$, then $\tilde{S}=\tilde{T}\res \alpha$ for some
Borel function $\alpha$ satisfying $0\leq \alpha\leq 1$. Therefore,
$\alpha=\lambda\circ \sigma$, which implies
$S= T\res \lambda$.
Analogously one gets $\partial \tilde{S} \leq \partial \tilde{T}= \delta_a-\delta_b$, which is only possible if $\alpha$ is constant over $[a, b]$ (minding that $\tilde{S}=\ld a, b\rd\res \alpha$).
Hence, also $\lambda$ is constant, and this completes the proof.
\end{proof}

\bibliographystyle{plain}
%\bibliography{mathopt} %% hard-coded:

\end{document}